\documentclass[11pt,a4paper]{article}

\usepackage{amsmath}
\usepackage{amsfonts}
\usepackage{amssymb}
\usepackage{amsmath}
\usepackage{amsthm}
\usepackage{latexsym}
\usepackage{graphicx,cite}
\DeclareGraphicsExtensions{.jpg}
\usepackage{a4wide}
\usepackage{float}
\usepackage{mathrsfs}
\usepackage[titletoc,title]{appendix}

\usepackage[dvipsnames]{xcolor}
\usepackage[colorlinks, citecolor=violet,linkcolor=red]{hyperref}

\usepackage{mathtools}
\mathtoolsset{showonlyrefs}

\newtheorem{thm}{Theorem}[section]
\newtheorem{lem}[thm]{Lemma}
\newtheorem{prop}[thm]{Proposition}
\newtheorem{cor}[thm]{Corollary}

\newtheorem*{mulone}{Multiplicity--One Conjecture\!}

\theoremstyle{definition}
\newtheorem{defn}[thm]{Definition}

\newtheorem{prob}[thm]{Problem}

\theoremstyle{remark}
\newtheorem{rem}[thm]{Remark}

\newtheorem{ex}[thm]{Example}

\numberwithin{equation}{section}

\def\dert{\partial_t}
\def\ders{\partial_s}
\def\comp{\circ}

\def\limup{\operatornamewithlimits{\overline{\lim}}}
\def\loc{_{\operatorname{\mathrm{loc}}}}

\def\HH{\mathcal H}

\def\R{\mathbb R}

\def\R{{{\mathbb R}}}

\def\NN{\mathbb N}

\def\ssss{\mathfrak s}
\def\tt{\mathfrak t}

\def\res  {
  \begin{picture}(9,8)
    \put (1,0){\line(0,1){8}}
    \put (1,0){\line(1,0){5}}
  \end{picture}}

\newcommand{\Om}        {\Omega}

\newcommand{\intbar}{\etaathop{\int\etaakebox(-13.5,0){\rule[4pt]{.7em}{0.3pt}}
\kern-6pt}\nolimits}

\newcommand{\be}{\begin{equation}}
\newcommand{\ee}{\end{equation}}
\newcommand{\bea}{\begin{equation*}}
\newcommand{\eea}{\end{equation*}}

\renewcommand{\t }{\tau }

\usepackage{tikz,fp,ifthen,fullpage}
\usetikzlibrary{backgrounds}
\usetikzlibrary{decorations.pathmorphing,backgrounds,fit,calc,through}
\usetikzlibrary{arrows}
\usetikzlibrary{shapes,decorations,shadows}
\usetikzlibrary{fadings}
\usetikzlibrary{patterns}
\usetikzlibrary{mindmap}
\usetikzlibrary{decorations.text}
\usetikzlibrary{decorations.shapes}

\title{Lectures on curvature flow of networks}
\author{Carlo Mantegazza \footnote{Dipartimento di Matematica e Applicazioni, Universit\`a di Napoli Federico II, Via Cintia, Monte S. Angelo
80126 Napoli, Italy} \and Matteo Novaga \footnote{Dipartimento di Matematica, Universit\`a di Pisa, Largo
    Bruno Pontecorvo 5, 56127 Pisa, Italy} \and 
    Alessandra Pluda\footnotemark[2]
}

\date{}

\begin{document}

\maketitle

\begin{abstract}
We present a collection of results on the evolution by curvature of networks of planar curves. We discuss in particular the existence of a solution and the analysis of singularities.
\end{abstract}


\section{Introduction}

These notes have been prepared for a course given by the second author within the IndAM Intensive Period {\it Contemporary Research in elliptic PDEs and related topics}, organized by Serena Dipierro at the University of Bari from April to June 2017. We warmly thank the organizer for the invitation, the IndAM for the support, and the Department of Mathematics of the University of Bari for the kind hospitality.

\medskip

The aim of this work is to provide an overview on the motion by curvature of a network of curves in the plane. 
This evolution problem attracted the attention of several researchers in recent years, see for instance~\cite{brakke,bronsard,kinderliu,Ilnevsch,haettenschweiler,belnov,chenguo,mannovplu,mannovtor,MMN13,pluda,schnurerlens,tonwic}.
We refer to the extended survey \cite{mannovplusch} for a motivation and a detailed analysis of this problem. 

This geometric flow can be regarded as the $L^2$-gradient flow of the length functional, which is the sum of the lengths of all the
curves of the network (see~\cite{brakke}).
From the energetic point of view it is then natural to expect that configurations with multi--points of order greater than three or $3$--points with angles different from $120$ degrees, being unstable for the length functional, should be present only at a discrete set of times, during the flow. Therefore, we shall restrict our analysis to networks whose junctions are composed by exactly three curves, meeting at 120 degrees. This is the so-called \textbf{Herring condition}, and we call \textbf{regular} the networks satisfying this condition at each junction.

The existence problem for the curvature flow of a regular network with only one triple junctions was first considered by L.~Bronsard and F.~Reitich in~\cite{bronsard}, where they proved the local existence of the flow, and by D.~Kinderlehrer and C.~Liu in~\cite{kinderliu}, who showed the global existence and convergence of a smooth solution if the initial network is sufficiently close to a minimal configuration (Steiner tree).  

We point out that the class of regular networks is not preserved by the flow, since two (or more) triple junctions might collide during the evolution, creating a multiple junction composed by more than three curves.
It is then natural to ask what is the subsequent evolution of the network. 
A possibility is 
restarting the evolution at the collision time with a different set of curves, describing a non--regular network, with multi--points of order higher than three.
A suitable short time existence result has been worked out by T.~Ilmanen,  
A.~Neves and F. Schulze in~\cite{Ilnevsch}, where it is shown that there exists a flow of networks which becomes immediately regular for positive times.

\medskip

These notes are organizes as follows: In Section~\ref{notation} we introduce the notion of
regular network and the geometric evolution problem we are interested in. 
In Section~\ref{existence} we recall the short time existence and uniqueness result by Bronsard and Reitich, and we sketch its proof.
We also show that the embeddedness of the network is preserved by the evolution
(till the maximal time of smooth existence).
In Section~\ref{selfsimilar} we describe some special solution which evolve self-similarly. More precisely, we discuss translating, rotating and homotetically shrinking solutions. The latter ones are particularly important for our analysis since they describe the blow-up limit of the flow near a singularity point. In Section~\ref{estimates} we derive the evolution equation for the $L^2$-norm of the curvature and of its derivatives. As a consequence, we show that, at a singular point, either the curvature blows-up or there is a collision of triple junctions.
Finally, in Section \ref{singular} we recall Huisken's Monotonicity Formula for mean curvature flow, which holds also for the evolution of a network, and we introduce the rescaling procedures used to get blow--up limits at the maximal time of smooth existence, in order to describe the singularities of the flow. In particular, we show that the limits of the rescaled networks are self-similar shrinking solutions of the flow, possibly with multiplicity greater than one, and we identify all the possible limits 
under the assumption that the length of each curve of the network is uniformly bounded from below.

\section{Notation and setting of the problem}\label{notation}

\subsection{Curves and networks}

Given an interval $I\subset\mathbb{R}$, 
we consider planar curves $\gamma:I\to\mathbb{R}^2$.

The interval $I$ can be both bounded and unbounded 
depending whether one wants to parametrize a bounded or an unbounded curve.
In the first case we restrict to consider $I=[0,1]$.

By \textbf{curve} we mean both image of the curve in $\mathbb{R}^2$ 
and parametrization of the curve, we will be more specific only when
the meaning cannot be got by the context.

\begin{itemize}
\item A curve is of class $C^k$ if it admits a parametrization 
$\gamma:I\to\mathbb{R}^2$ of class $C^k$.
\item A $C^1$ curve, 
is \textbf{regular} if  it admits a regular parametrization, namely
$\gamma_x(x)=\frac{d\gamma}{dx}(x)\neq 0$  for every $x\in I$. 
\item It is then well defined
its \textbf{unit tangent vector} $\tau=\gamma_{x}/|\gamma_{x}|$. 
\item We define its \textbf{unit normal vector} as
$\nu=R\tau=R\gamma_{x}/|\gamma_{x}|$, 
where $R:\mathbb{R}^{2}\to\mathbb{R}^{2}$ is the anticlockwise
rotation centred in the origin of $\mathbb{R}^{2}$ of angle
${\pi}/{2}$.
\item The  \textbf{arclength parameter} of a curve $\gamma$ is given by
$$
s:=s(x)=\int_0^x\vert\gamma_x(\xi)\vert\,d\xi\,.
$$
We use the letter $s$ to indicate
the arclength parameter and the letter $x$
for any other parameter.
Notice that $\partial_s=\vert\gamma_x\vert^{-1}\partial_x$.
\item If the curve $\gamma$ is $C^2$ and regular, 
we define the curvature $k:=\vert \tau_s\vert=\vert \gamma_{ss}\vert$ and 
the \textbf{curvature  vector} $\boldsymbol{k}:=\tau_s=\gamma_{ss}$.
We get: 
$$
\boldsymbol{k}
=\frac{1}{|\gamma_{x}|}\left(\frac{\gamma_{x}}{|\gamma_{x}|} \right)_x
=\frac{\gamma_{xx}\vert \gamma_x\vert^2-\gamma_x 
\langle \gamma_{xx},\gamma_x \rangle}{\vert \gamma_x\vert^4}\,.
$$
As we are in $\mathbb{R}^2$ we remind that  $\boldsymbol{k}=\tau_s=k\nu$.
\item The \textbf{length} $L$ of a curve $\gamma$ is given by 
$$
L(\gamma):=\int_{I}\vert \gamma_x(x)\vert \,dx=\int_\gamma 1\,ds\,.
$$
\end{itemize}

A curve is injective if for every $x\neq y\in I$ we have $\gamma(x)\neq\gamma(y)$.\\
A curve $\gamma:[0,1]\to\mathbb{R}^2$ of class $C^k$ is 
\textbf{closed} if $\gamma(0)=\gamma(1)$ and if $\gamma$ has a $1$--periodic $C^k$ 
extension to $\mathbb{R}$ (Figure~\ref{closedcurve}).

\begin{figure}[H]
\begin{center}
\begin{tikzpicture}[scale=0.78]
\draw
(-4.25,0) to[out=90,in=-100,looseness=1] (-2,1)
(-2,0.7) node[below] {$\gamma$} (-2,1)
to[out=80,in=90,looseness=1](0,0)
to[out=270,in=60,looseness=1](-2.5,-1)
to[out=-120,in=270,looseness=1](-4.25,0);
\end{tikzpicture}
\end{center}
\begin{caption}{A simple closed curve.}\label{closedcurve}
\end{caption}
\end{figure}

In what follows we will consider 
time--dependent families of curves 
$(\gamma(t,x))_{t\in[0,T]}$.
We let $\tau=\tau\left(t,x\right)$ be the unit
tangent vector to the curve, $\nu=\nu\left(t,x\right)$ the unit
normal vector and
$\boldsymbol{k}=\boldsymbol{k}\left(t,x\right)$
its curvature vector as previously defined.

We denote with $\partial_xf$, $\partial_sf$ and $\partial_tf$
the derivatives of a function $f$ along a curve $\gamma$ with respect to the $x$ variable, 
the arclength parameter $s$ on such curve and the time, respectively.
Moreover
$\partial^n_xf$, $\partial^n_sf$, $\partial^n_tf$ are
the higher order partial derivatives, possibly denoted also by 
$f_x,f_{xx}\dots$, $f_{s}, f_{ss},\dots$ and  $f_t, f_{tt},\dots$.

We adopt the following convention for integrals:
$$
\int_{{\gamma_t}} f(t,\gamma,\tau,\nu,k,k_s,\dots,\lambda,\lambda_s\dots)\,ds =
 \int_0^1
f(t,\gamma^i,\tau^i,\nu^i,k^i,k^i_s,\dots,\lambda^i,\lambda^i_s\dots)\,\vert
\gamma^i_x\vert\,dx
$$
as the arclength measure is given by $ds=\vert\gamma^i_x\vert\,dx$ on every curve $\gamma$.

\medskip


Let now $\Omega$ be a smooth, convex, open set in $\mathbb{R}^{2}$.

\begin{defn}\label{network}
 A \textbf{network} $\mathcal{N}$ in
$\overline{\Omega}$ is a connected set described by a finite family
of  regular $C^1$ curves contained in $\overline{\Omega}$ such that
\begin{enumerate}
\item the interior of every curve is injective,
a curve can self--intersect only at its end--points;
\item two different curves can intersect each other only at their end--points;
\item a curve is allowed to meet $\partial\Omega$ only at its end--points;
\item if an end--point of a curve coincide with $P\in\partial\Omega$,
then no other end--point of any curve can coincide with $P$.
\end{enumerate}
The curves of a network can meet at 
\textbf{multi--points} in $\Omega$, labeled by
$O^{1}, O^{2},\dots, O^{m}$.
We call \textbf{end--points} of the network, the vertices (of order one)
$P^{1}, P^{2},\dots, P^{l}\in\partial\Omega$.
\end{defn}

Condition~4  keeps things simpler implying that  multi--points can be 
only inside $\Omega$, not on the boundary.

We say that a network is of class $C^{k}$ with $k\in\{1,2,\ldots\}$ if all its 
curves are of class $C^{k}$.

\begin{rem}
With a slightly modification of Definition~\ref{network} we could also consider
networks in the whole $\mathbb{R}^2$ with unbounded curves.
In this case we require that
every non compact branch of $\mathcal{N}$ 
is asymptotic to an half line and its curvature 
is uniformly bounded.
We call these unbounded networks \textbf{open networks}. 
\end{rem}

\begin{defn}\label{regularnetwork}
We call a network \textbf{regular} if
all its multi--points  are triple and 
 the sum of unit tangent vectors of the 
concurring curves
at each of them
 is zero.
\end{defn}

\begin{ex}\label{esempisemplici}$\,$
\begin{itemize}
\item A network could consists of a single closed embedded curve.
\item A network could be composed of a single embedded curve with fixed end--points
on $\partial\Omega$. 
\item  There are two possible  (topological) structures of networks with only one triple
junction: the  \textbf{triod} $\mathbb{T}$ or the  \textbf{spoon} $\mathbb{S}$.
A triod is a tree  composed of three curves
that intersects each other at a  $3$--point 
and have their other  end--points on the boundary of $\Omega$.
A spoon  is the union of two 
curves: a closed one attached to the other  at a triple junction. 
The  ``open'' curve  of the spoon has an end--point on $\partial\Omega$
(Figure~\ref{triodspoon}).

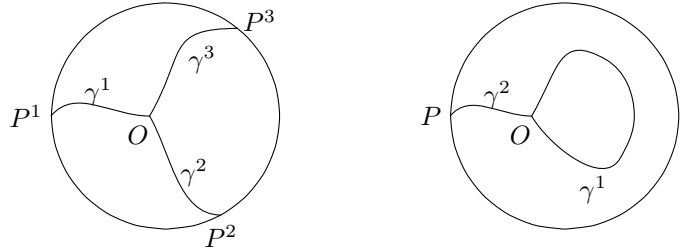
\begin{figure}[h]
\begin{center}
\begin{tikzpicture}[scale=0.75]
\draw (-3.73,0) 
to[out= 50,in=180, looseness=1] (-2,0) 
to[out= 60,in=180, looseness=1.5] (-0.45,1.55) 
(-2,0)
to[out= -60,in=180, looseness=0.9] (-0.75,-1.75);
\draw[color=black,scale=1,domain=-3.141: 3.141,
smooth,variable=\t,shift={(-1.72,0)},rotate=0]plot({2.*sin(\t r)},
{2.*cos(\t r)}) ;
\path[font=\small]
    (-3.7,0) node[left]{$P^1$}
    (-2.9,0.85) node[below] {$\gamma^1$}
    (-1.5,1) node[right] {$\gamma^3$}
    (-0.8,-1)[left] node{$\gamma^2$}
    (-2.2,0) node[below] {$O$}
    (-0.07,1.32)node[above]{$P^3$}
    (-0.75,-1.75) node[below] {$P^2$};

\draw[shift={(7,0)}] 
(-3.73,0) 
to[out= 50,in=180, looseness=1] (-2.3,0) 
to[out= 60,in=150, looseness=1.5] (-1,1) 
(-2.3,0)
to[out= -60,in=-120, looseness=0.9] (-0.75,-0.75)
(-1,1)
to[out= -30,in=90, looseness=0.9] (-0.5,0)
to[out= -90,in=60, looseness=0.9] (-0.75,-0.75);
\draw[color=black,scale=1,domain=-3.141: 3.141,
smooth,variable=\t,shift={(5.28,0)},rotate=0]plot({2.*sin(\t r)},
{2.*cos(\t r)}) ;
\path[font=\small, shift={(7,0)}]

   (-3.7,0) node[left]{$P$}
    (-2.9,0.8) node[below] {$\gamma^2$}
     (-0.8,-1.3)[left] node{$\gamma^1$}
    (-2.5,0) node[below] {$O$};     
\end{tikzpicture}
\end{center}
\caption{A triod and a spoon.}\label{triodspoon}
\end{figure}

\end{itemize}
\end{ex}

\subsection{The evolution problem}

Given a network composed of $n$ curves we define its
\textbf{global length} as
 $$
 L = L^1+ \dots + L^n\,.
 $$
The evolution we have in mind is the $L^2$--gradient flow of the global length $L$.
Therefore, geometrically
speaking, this means that the normal velocity of the curves is the curvature.
In the case of the curves (curve shortening flow)
this condition fully defines  the evolution, at least geometrically.
In the case of networks another condition at the junctions comes from
the variational formulation of the evolution, as we will see below.

\subsubsection{Formal derivation of the gradient flow }

We begin by considering one closed embedded $C^2$ curve, parametrized by 
$\gamma:[0,1]\to\mathbb{R}^2$.
Then 
$\gamma(0)=\gamma(1)$, $\tau(0)=\tau(1)$ and $k(0)=k(1)$.
We want to compute the directional derivative of the length.
Given $\varepsilon\in\mathbb{R}$ and $\psi:[0,1]\to\mathbb{R}^2$ 
a smooth function satisfying $\psi(0)=\psi(1)$, we take
$\widetilde{\gamma}=\gamma+\varepsilon\psi$
a variation of $\gamma$. From now on we neglect the dependence on the variable $x$
to maintain the notation simpler.
We have  
\begin{align*}
\frac{\partial}{\partial \varepsilon}L(\widetilde{\gamma})_{\vert _{\varepsilon=0}}&=
\frac{\partial}{\partial \varepsilon} \int_0^1\vert \gamma_x+\varepsilon\psi_x\vert\, dx=
\int_0^1\frac{\left\langle \psi_x,\gamma_x \right\rangle }{\vert \gamma_x\vert}\,dx
=\int_\gamma\left\langle\psi_s,\tau\right\rangle\,ds\\
=&-\int_\gamma\left\langle \psi,\tau_s\right\rangle\,ds+
\left\langle \psi(1), \tau(1)\right\rangle -\left\langle \psi(0), \tau(0)\right\rangle\,.
\end{align*}
As $\gamma$ is a simple closed embedded curve, then the boundary terms are equal zero.
We get  
$$
\frac{\partial}{\partial \varepsilon}L(\widetilde{\gamma})_{\vert _{t=0}}=
\int_0^1\left\langle \psi, -\boldsymbol{k} \right\rangle\, ds\,.
$$

Since we have written the directional derivative of $L$ in the direction $\psi$ as the scalar product
of $\psi$ and $-\boldsymbol{k}$, we conclude (at least formally) that
$-\boldsymbol{k}$ is the gradient of the length.
Hence we can understand  the curve shortening flow as the gradient flow of the length.

\medskip

We considering now a triod $\mathbb{T}$ in a convex, 
open and regular set $\Omega\subset\mathbb{R}^2$,
whose curves are
parametrized by $\gamma^i:[0,1]\to\mathbb{R}^2$ of class $C^2$
with $i\in\{1,2,3\}$. Without loss of generality we can suppose that 
$\gamma^1(0)=\gamma^2(0)=\gamma^3(0)$ and $\gamma^i(1)=P^i\in\partial\Omega$ 
with $i\in\{1,2,3\}$.
We consider again a variation $\tilde{\gamma}^i=\gamma^i+\varepsilon\psi^i$
of each curve with 
$\psi^i:[0,1]\to\mathbb{R}^2$ three smooth functions.
We require that $\psi^1(0)=\psi^2(0)=\psi^3(0)$ and $\psi^i(1)=0$ because
we want that the set $\tilde{\mathbb{T}}$ parametrized by 
$\tilde{\gamma}=(\tilde{\gamma}^1,\tilde{\gamma}^2,\tilde{\gamma}^3)$ is a triod
with end point on $\partial\Omega$ fixed at $P^i$.
In such a way we are 
asking two (Dirichlet) boundary conditions.
By definition of total length $L$ of a network, we have
$$
L(\widetilde{\mathbb{T}})=\sum_{i=1}^3 L(\mathcal{\gamma}^i)
=\sum_{i=1}^3 \int_0^1 \vert \widetilde{\gamma}^i_x\vert\, dx
=\sum_{i=1}^3 \int_0^1\vert \gamma^i_x+\varepsilon\psi^i_x\vert\, dx\,.
$$
Repeating the previous computation and using the hypothesis on $\psi^i$ we have
\begin{align*}
\frac{\partial}{\partial \varepsilon}L(\widetilde{\mathbb{T}})_{\vert _{\varepsilon=0}}
&=\sum_{i=1}^3\int_\gamma\left\langle \psi^i, -\boldsymbol{k}^i\ \right\rangle\, ds+
\sum_{i=1}^3
\left\langle \psi^i(1), \tau^i(1)\right\rangle   
-\sum_{i=1}^3\left\langle \psi^i(0), \tau^i(0)\right\rangle\\
&=\sum_{i=1}^3\int_\gamma\left\langle \psi^i, -\boldsymbol{k}^i\ \right\rangle\, ds+
\sum_{i=1}^3
-\left\langle \psi^1(0), \tau^i(0)\right\rangle \,.
\end{align*}
Imposing that the boundary term equals zero 
we get  
$$
0=\sum_{i=1}^3\left\langle \psi^1(0),\tau^i(0)\right\rangle
=
\left\langle \psi^1(0), \sum_{i=1}^3\tau^i(0)\right\rangle
\Longrightarrow\sum_{i=1}^3\tau^i(0)=0\,.
$$
Hence, we have derived a further boundary condition at the junctions.

\subsubsection{Geometric problem}

We define the \textbf{motion by curvature} of regular networks.

\begin{prob}\label{probgradientflow}
Given a regular network 
we let it evolve by the $L^2$--gradient flow of the (total) 
length functional $L$ in a maximal time interval $[0,T)$.
That is:
\begin{itemize}
\item each curve of the network has a normal velocity equal to its curvature 
at every point and  for all times $t\in [0,T)$
-- \textbf{motion by curvature};
\item the curves that meet at junctions remains attached for all times $t\in [0,T)$ --
\textbf{concurrency};
\item the sum of the unit tangent vectors of the three curves meeting at a junction
is zero  for all times $t\in [0,T)$ --
\textbf{angle condition}.
\end{itemize}
Moreover we ask that
the end--points
$P^{r}\in\partial\Omega$ stay fixed during the evolution
-- \textbf{Dirichlet boundary
  condition}.
\end{prob}

 As a possible variant one  lets
the end--points free to move on the boundary of $\Omega$ but
asking that the curves intersect orthogonally $\partial\Omega$ 
-- \textbf{Neumann boundary condition}.

\medskip

Although our problem is \textbf{geometric} (as we want to describe the flow of a set
moving in $\mathbb{R}^2$), to solve we will turn to a \textbf{parametric approach}.
As a consequence we will work often at the level of parametrization.

\begin{defn}[Geometric admissible initial data]
A network
$\mathcal{N}_0$ is a geometrically admissible initial data
for the motion  by curvature
if it is regular,
at each junction the sum of the curvature is zero, the curvature at each end--point
on $\partial\Omega$ is zero and 
each of its 
curve can be parametrized by a regular curve $\gamma_0^i:[0,1]\to\mathbb{R}^2$
of class $C^{2+\alpha}$ with $\alpha\in (0,1)$. 
\end{defn}

We introduce a way to label the curves:
given a network composed by $n$ curves 
with $l$ end--points $P^1, P^2,\dots, P^l\in\partial\Omega$ (if
present) and $m$
triple points $O^1, O^2,\dots O^m\in\Omega$, we denote with
$\gamma^{pi}$ the curves of this network concurring at the
multi--point $O^p$ with
$p\in{\{1,2,\dots,m\}}$ and $i\in{\{1,2,3\}}$.

\begin{defn}\label{geosolution}[Solution of the motion by curvature of networks]
Consider a geometrically admissible initial network $\mathcal{N}_0$ 
 composed of $n$ curves parametrized by 
$\gamma_0^i:[0,1]\to\overline{\Omega}$,
with $m$ triple points $O^1, O^2,\dots
  O^m\in\Omega$ and (if present) $l$ end--points $P^1, P^2,\dots, P^l\in\partial\Omega$.
A time dependent
family of networks $\left(\mathcal{N}_t\right)_{t\in[0,T)}$
is a solution of the motion by curvature in the maximal time interval $[0,T)$ 
with initial data $\mathcal{N}_0$ if it admits a time dependent family of  parametrization
$\gamma=(\gamma^1,\ldots,\gamma^n)$ such that 
each curve $\gamma^i\in C^{\frac{2+\alpha}{2},2+\alpha}([0,T)\times[0,1])$
is regular and
the following system of conditions is satisfied
for every $x\in[0,1]$, $t\in[0,T)$,  $i,j\in\{1,2,\dots, n\}$
\begin{equation}\label{problemageometrico}
\begin{cases}
\begin{array}{lll}
(\gamma^i)^\perp_t(t,x)=\boldsymbol{k}^i(t,x) &&\text{ motion by curvature,}\\
\gamma^{pi}=\gamma^{pj} & \text{at every $3$--point $O^p$} &\text{ concurrency,}\\
\sum_{i=1}^3\tau^{pi}=0\quad& \text{at every $3$--point $O^p$}
&\text{ angle condition,}\\
\gamma^r(t,1)=P^r\quad&\text{with}\; 0\leq r \leq l\,
\quad &\text{ Dirichlet boundary condition,}\\
\end{array}
\end{cases}
\end{equation}
where we assumed conventionally that the end--point
$P^r$ of the network is given by $\gamma^r(t,1)$.
\end{defn}

\begin{rem}
The boundary conditions in system~\eqref{problemageometrico}
are consistent with a second order flow of three curves.
Indeed we expect three vectorial conditions at  the junctions and one for each curve
at the other end points.
\end{rem}

\begin{rem}
We have defined  solutions in $C^{\frac{2+\alpha}{2},2+\alpha}$
but 
the natural class seems to be $C^{1,2}$.
It is indeed possible to define a solution to the motion by curvature of networks 
asking less regularity on the parametrization.
Our choice  simplify the proof of the short time existence result. We will see in the
sequel that it is based on linearization and on a fixed point argument. 
The classical theory for system of linear parabolic equations 
developed by Solonnikov~\cite{solonnikov1} is a H\"{o}lder functions
setting (see~\cite[Theorem 4.9]{solonnikov1}).
\end{rem}

\begin{rem}\label{geocomprem1}
Suppose that $(\mathcal{N}(t))_{t\in[0,T]}$
is a solution to the motion by curvature as defined in~\ref{geosolution}.
We will see later that at $t>0$ the curvature at the end--points and the sum of the three curvatures at every $3$--point are automatically zero.
Then a necessary condition for $(\mathcal{N}(t))_{t\in[0,T]}$ to be $C^{2}$ in
space till $t=0$  is that these properties are satisfied
also by the the initial regular network. 
These conditions on the curvatures are \textbf{geometric},
independent of the parametrizations of the curves, but intrinsic to the \textbf{set} 
and they are not satisfied by a generic regular, $C^2$ network.
\end{rem}

\begin{rem}
Notice that in the geometric problem we specify only the  
the \textbf{normal component} of the velocity of the curves (their curvature).
This does not mean that there is not a tangential component of the velocity, 
rather a tangential motion is needed to allow the junctions move in any direction.
\end{rem}

\begin{ex}\label{gremh}$\,$
\begin{itemize}
\item The motion by curvature of a single closed embedded curve
was widely studied by many authors~\cite{angen1,angen2,angen3,gage,gage0,gaha1,gray1, manmaggra}.
In particular the curve evolves smoothly, becoming convex
and getting rounder and rounder. In finite time it
shrinks  to a point. 
\item The case of a curve with either an angle or a cusp 
can be dealt by the works of Angenent~\cite{angen1,angen2,angen3}.
Actually the  curve becomes immediately smooth and then for all positive time we come back 
to the evolution described in the previous example.
\item The evolution of  a single embedded curve with fixed
 end--points (Figure~\ref{onecurve}) is discussed in~\cite{huisk2,stahl1,stahl2}.
The curve converges to the straight segment connecting the two fixed end--points
as the time goes to infinity.
\item Two curves that concur at a $2$--point 
forming an angle (or a cusp, if they have the same tangent)
can be regarded as a single curve with a singular point,
which will vanish immediately under the flow (Figure~\ref{onecurve}).
\end{itemize}

\begin{figure}[h]
\begin{center}
\begin{tikzpicture}[scale=0.78]
\draw [shift={(-2.5,0)}]
(3.75,0) node[left] {$P$}
to[out=30,in=110, looseness=1] (4.97,0) 
(4.5,0.3) node[below] {$\sigma^1$} (4.97,0)
to[out= -70,in=180, looseness=1] (6.19,0)
node[below] {$O$}
to[out=60, in=-145, looseness=1] (7, 0.81)
(7.2,0.9) node[below] {$\sigma^2$} (7,0.81)
to[out=35,in=150,  looseness=1] (8.62,0.81)
node[right] {$Q$};
\draw[white]
(-1,-1.5)--(1,-1.5);
\end{tikzpicture}
\quad\quad\quad
\begin{tikzpicture}[scale=0.53]
\draw[color=black,scale=1,domain=-3.141: 3.141,
smooth,variable=\t,shift={(-1,0)},rotate=0]plot({3.25*sin(\t r)},
{2.5*cos(\t r)});
\draw  
(-1.81,0)
to[out=90, in=-145, looseness=1] (-1, 0.81)
to[out=35,in=150,  looseness=1] (0.62,0.81)
to[out=-30,in=150, looseness=1.5] (1.95,1)
node[right]{$\, P^2$}
(1.62,-1.5) node[right]{$P^1$}
to[out=130,in=20, looseness=1] (0.62,-0.81) 
to[out=-160,in=-90,  looseness=0.5] (-1.81,0);
\path[font= \large]
(-3.75,-1.8) node[below] {$\Omega$};
\path (-2.3,0.5) node[right] {$\sigma$};
\end{tikzpicture}
\end{center}
\begin{caption}{Two special cases: two curves forming an angle at their junction
and a single curve with two end--points on the boundary of $\Omega$.}\label{onecurve}
\end{caption}
\end{figure}
\end{ex}

\subsubsection{The system of quasilinear PDEs}

In this section
we actually work by defining the evolution in terms of differential equations
for the parametrization of the curves.
For sake of presentation
we restrict to the case of the triod.
This allows us maintaining the notation simpler.

\medskip

Let us start focusing on the geometric evolution equation 
$\gamma_t^\perp=\boldsymbol{k}$,
that can be equivalently written as
\begin{equation}\label{evoleqR2}
\left\langle\gamma_t(t,x)\,,\,\nu(t,x)\right\rangle
\nu(t,x)
=\left\langle\frac{\gamma_{xx}(t,x)}{{\vert{\gamma_x(t,x)}\vert}^2}\,,\,\nu(t,x)\right\rangle
\nu(t,x)\,.
\end{equation} 
This equation specify the velocity of each curve only in direction of
the normal $\nu$.

\medskip

Curve shortening flow for closed curve is not affected by tangential velocity.
In the evolution by curvature of a smooth closed curve it is well known 
that any tangential contribution to the velocity actually affects only the ``inner motion'' 
of the ``single points'' (Lagrangian point of view), 
but it does not affect the motion of the whole curve as a subset of $\mathbb{R}^2$
(Eulerian point of view). 
Indeed
the classical mean curvature flow for hypersurfaces 
is invariant under tangential perturbations
(see for instance~\cite[Proposition~1.3.4]{Manlib}).
In particular in the case of curves
it can be shown that a solution of the curve shortening flow 
satisfying the equation
$\gamma_t=k\nu+\lambda\tau$ for some continuous function $\lambda$ 
can be globally reparametrized (dynamically in time) in order to satisfy $\gamma_t=k\nu$
and vice versa.

\medskip

As already anticipated, 
in the case of networks it is instead \textbf{necessary} to consider an extra \textbf{tangential term}
(as for the case of the single curve that is not closed). 
It allows the motion of the $3$--points.
At the junctions the sum of
the unit normal vectors is zero.
If the velocity would be in normal direction to the three curves concurring at a 3--point, 
this latter should move in a
direction which is normal to all of them,
then the only possibility would be that the junction does not move at all.

\smallskip

Saying that a junction cannot move is equivalent to fix it, hence to add a condition
in the system~\eqref{problemageometrico}.
Thus, from the PDE point of view, the system
becomes overdetermined as at the junctions we have already required the concurrency and the 
angle conditions. 

\medskip

Therefore solving the problem of the motion by curvature of regular networks
means that we require the concurrency and the angle condition
(regular networks remain regular networks for all the times)
 and that the main equation for each curve is
$$
\gamma^i_t(t,x)=k^i(t,x)\nu^i(t,x)+\lambda^i(t,x)\tau^i(t,x)
$$
for some $\lambda$ continuous function not specified.
To the aim of writing a non--degenerate PDE for each curve 
we consider the tangential velocity 
$$
\lambda^i=\frac{\left\langle\gamma^i_{xx}\vert\tau^i\right\rangle}{\vert\gamma_x^i\vert^2}\,.
$$

Then the  velocity of the curves is 
\begin{equation}\label{special}
\gamma^i_t(t,x)
=\left\langle\frac{\gamma_{xx}^i(t,x)}{{\vert{\gamma_x^i(t,x)}\vert}^2}\,\Big\vert\,\nu^i(t,x)\right\rangle
\nu^i(t,x)+\left\langle\frac{\gamma_{xx}^i(t,x)}{{\vert{\gamma_x^i(t,x)}\vert}^2}\,\Big\vert\,\tau^i(t,x)\right\rangle
\tau^i(t,x)=\frac{\gamma^i_{xx}(t,x)}{\vert \gamma^i_x(t,x\vert^2}\,.
\end{equation} 
A family of networks evolving according to \eqref{special} will be called  
a \textbf{special flow}.

\medskip

We are finally able to write
explicitly the system of PDE we consider.

Without loss of generality
any triod $\mathbb{T}$ can be 
parametrized by $\gamma=(\gamma^1,\gamma^2,\gamma^3)$ 
in such a way that 
the triple junction is $\gamma^1(0)=\gamma^2(0)=\gamma^3(0)$
and that the other end--points $P^i$ on $\partial \Omega$
are given by $\gamma^i(1)=P^i$ with $i\in\{1,2,3\}$.

\begin{defn}\label{probdef} 
Given an admissible initial parametrization $\varphi=(\varphi^1,\varphi^2,\varphi^2)$ of a geometrically admissible
initial triod $\mathbb{T}_0$
the family of time--dependent parametrizations
$\gamma=(\gamma^1,\gamma^2,\gamma^3)$ is a 
\textbf{solution of the special flow}  in the time interval $[0,T]$
if the functions $\gamma^i$ are of class $C^{\frac{2+\alpha}{2},2+\alpha}([0,T]\times[0,1])$
and
the following system is satisfied 
for every $t\in[0,T]$,  $x\in[0,1]$,   $i\in\{1,2,3\}$
\begin{equation}\label{problema}
\begin{cases}
\begin{array}{lll}
\gamma^i_t(t,x)=\frac{\gamma^i_{xx}(t,x)}{\vert\gamma^i_x(t,x)\vert^2}\quad&
&\text{ motion by curvature,}\\
\gamma^1(t,0)=\gamma^2(t,0)=\gamma^3(t,0)\quad&
&\text{ concurrency,}\\
\sum_{i=1}^3\tau^{i}(t,0)=0\quad&
\quad &\text{ angle condition,}\\
\gamma^i(t,1)=P^i\quad& &\text{ Dirichlet boundary condition}\\
\gamma^i(0,x)=\varphi^i(x)\quad& &\text{ initial data}
\end{array}
\end{cases}
\end{equation}
\end{defn}

\begin{defn}[Admissible initial parametrization of a triod]\label{admissible}
We say that 
a parametrization $\varphi=(\varphi^1,\varphi^2,\varphi^3)$
is admissible for the system~\eqref{problema} if:
\begin{enumerate}
\item $\cup_{i=1}^3\varphi^i([0,1])$ is a triod;
\item each curve $\varphi^i$ is regular and of class $C^{2+\alpha}([0,1])$;
\item $\varphi^i(0)=\varphi^j(0)$  for every $i,j\in\{1,2,3\}$;
\item $\frac{\varphi_x^1(0)}{\vert \varphi_x^1(0)\vert}
+\frac{\varphi_x^2(0)}{\vert \varphi_x^2(0)\vert}+
\frac{\varphi_x^3(0)}{\vert \varphi_x^3(0)\vert}=0$;
\item $\frac{\varphi_{xx}^i(0)}{\vert \varphi_x^i(0)\vert^2}=
\frac{\varphi_{xx}^j(0)}{\vert \varphi_x^j(0)\vert^2}$  for every $i,j\in\{1,2,3\}$;
\item $\varphi^i(1)=P^i$ for every $i\in\{1,2,3\}$;
\item $\varphi_{xx}^i(1)=0$  for every $i\in\{1,2,3\}$.
\end{enumerate}
\end{defn}

\begin{rem}
Notice that in the literature one refers to conditions $3.$ to $7.$
in Definition~\ref{admissible} as  \textbf{compatibility conditions} for 
system~\eqref{problema}. In particular conditions $5.$ and $7.$
are called compatibility conditions of order $2$.
\end{rem}

We want to stress the fact that choosing the tangential velocity and so passing
to consider the special flow allows us to turn the geometric problem into a 
non degenerate PDE's system.
The goodness of our choice will be revealed
when one verifies the well posedness of the system~\eqref{problema}.

\smallskip

Once proved existence and uniqueness of solution for the PDE's system,
it is then crucial to come back to the geometric problem and show that we have solved it
in a ``geometrically" unique way.
This can be done in two step: first
one shows that for any geometrically admissible initial data
there exists an admissible initial parametrization for system~\eqref{problema} 
(and consequently a unique solution related to that parametrization).
In the second step one supposes that there exist two different 
solutions of the geometric problem
and then proves that it is possible to 
pass from one to another by time--dependent reparametrization.
 
 \smallskip
 
However from the previous discussion we have understood that 
in our situation of motion of networks
the invariance under tangential terms of the curve shortening flow
is not trivially true.
To prove existence and uniqueness of the motion by curvature
of networks starting from existence and uniqueness
of the PDE's system solution 
a key role will be played again by our good choice of the tangential velocity.

\section{Short time existence and uniqueness}\label{existence}

We now deal with the problem of short time existence and uniqueness  of the flow.

\subsection{Existence and uniqueness for the special flow}

We restrict again to a triod in $\Omega\subset\mathbb{R}^2$.
We consider first system~\eqref{problema}.
The short time existence result 
is due to Bronsard and Reitich~\cite{bronsard}.

\medskip

We look for classical solutions in the space
$C^{\frac{2+\alpha}{2},2
+{\alpha}}\left(\left[0,T\right]\times\left[0,1\right]\right)$ with $\alpha\in (0,1)$.
We recall the definition of this function space and of the norm it is endowed with
(see also~\cite[\S 11, \S 13]{solonnikov1}).

For a function $u:[0,T]\times [0,1]\to\mathbb{R}$ we define the semi--norms
$$
[ u]_{\alpha,0}:=\sup_{(t,x), (\tau,x)}\frac{\vert u(t,x)-u(\tau,x)\vert}{\vert t-\tau\vert^\alpha}\,,
$$
and
$$
[ u]_{0,\alpha}:=\sup_{(t,x), (t,y)}\frac{\vert u(t,x)-u(t,y)\vert}{\vert x-y\vert^\alpha}\,.
$$
The classical parabolic H\"older space
$C^{\frac{2+\alpha}{2}, 2+\alpha}([0,T]\times[0,1])$ is the space 
of all functions $u:[0,T]\times [0,1]\to\mathbb{R}$ that have continuous derivatives 
$\partial_t^i\partial_x^ju$ (where $i,j\in\mathbb{N}$ are such that $2i+j\leq 2$) for which the norm
\begin{equation*}
\left\lVert u\right\rVert_{C^{\frac{2+\alpha}{2},2+\alpha}}:=\sum_{2i+j=0}^2\left\lVert\partial_t^i\partial_x^ju\right\rVert_\infty
+\sum_{2i+j=2}\left[\partial_t^i\partial_x^ju\right]_{0,\alpha}+\sum_{0<2+\alpha-2i-j<2}\left[\partial_t^i\partial_x^ju\right]_{\frac{2+\alpha-2i-j}{2},0}
\end{equation*}
is finite.

The boundary terms are in spaces of the form
$C^{\frac{k+\alpha}{2}, k+\alpha}([0,T]\times\{0,1\}, \mathbb{R}^m)$
with $k\in\{1,2\}$
which we identify with 
$C^{\frac{k+\alpha}{2}}([0,T], \mathbb{R}^{2m})$
via the isomorphism
$f\mapsto (f(t,0),f(t,1))^t$.

\medskip

Calling $B_r$ the ball of radius $r$ centred at the origin the short time existence
result reads as follows:

\begin{thm}[Bronsard and Reitich]\label{2smoothexist0-triod}   
For any admissible initial parametrization there exists
a positive radius $M$ and a positive time $T$
such that the system~\eqref{problema} has a unique solution in
$C^{\frac{2+\alpha}{2},2+\alpha}\left(\left[0,T\right)\times\left[0,1\right]\right)\cap\overline{B}_M$.
\end{thm}

\begin{rem} 
Actually in~\cite{bronsard} the authors do not consider exactly system~\eqref{problema}, 
but the analogous Neumann problem. 
They require that the end--points of the three curves intersect the boundary of $\Omega$  with a prescribed angle 
(of 90 degrees).
\end{rem}

Bronsard and Reitich approach, based on linearising 
the problem around the
initial data,   nowadays is considered classical.
We explain here their strategy.

\medskip

\textit{Step 1: Linearization}

Fix an admissible initial datum $\sigma=(\sigma^1,\sigma^2,\sigma^3)$.
We  linearise the system~\eqref{problema} around $\sigma$ getting
\begin{align}\label{linmotion}
\gamma^i_t
-\frac{1}{\vert\sigma^i_x\vert^2}\gamma^i_{xx}
&=\left(\frac{1}{\vert\gamma^i_x\vert^2}-\frac{1}{\vert\sigma^i_x\vert^2}\right)\gamma^i_{xx} 
=:f^i(\gamma^i_{xx},\gamma^i_{x})\,.
\end{align}

The concurrency condition and  the Dirichlet boundary condition are already linear.
The angle condition instead is not linear, so one has to take into account the linear
version of it:
\begin{align}
-\sum_{i=1}^3 \frac{\gamma^i_x}{\vert \sigma^i_x\vert}
-\frac{\sigma^i_x\left\langle \gamma^i_x,\sigma^i_x\right\rangle}{\vert \sigma^i_x\vert^3}
&=\sum_{i=1}^3 \left(\frac{1}{\vert \gamma^i_x\vert}
 -\frac{1}{\vert\sigma^i_x\vert}\right)\gamma^i_x +
 \frac{\sigma^i_x\left\langle \gamma_x,\sigma^i_x\right\rangle}{\vert \sigma^i_x\vert^3}
=:b(\gamma_x) \,.
\end{align}

The linearized system associated to~\eqref{problema} is the following: 
for $i\in\{1,2,3\}$, $t\in[0,T]$ and $x\in[0,1]$
\begin{equation}\label{linsys}
\begin{cases}
\begin{array}{lll}
\gamma^i_t(t,x)-\frac{\gamma^i_{xx}(t,x)}{\vert\sigma^i_x\vert^2}
&=f^i(t,x) &\;\text{motion,}\\
\gamma^{1}(t,0)-\gamma^{2}(t,0)&=0 &\;\text{concurrency}\\
\gamma^{1}(t,0)-\gamma^{3}(t,0)&=0 &\;\text{concurrency}\\
-\sum_{i=1}^3\frac{\gamma_{x}^{i}(t,0)}{\left|\sigma_{x}^{i}\right|} -\frac{\sigma^i_x\left\langle \gamma(t,x)^i_x,\sigma^i_x\right\rangle}{\vert \sigma^i_x\vert^3}&=
b(t,0)&\;\text{angles condition}\\
\gamma^i(t,1)&=P^i &\;\text{Dirichlet boundary condition}\\
\gamma^{i}(0,x)&=\varphi^{i}(x) &\;\text{initial data}\\
\end{array}
\end{cases}
\end{equation}

We remind that the initial data for the system has to satisfy some linear compatibility conditions.

\medskip

\textit{Step 2: Existence and uniqueness of solution for the linearized system}

We have linearized system~\eqref{problema} to
obtain system~\eqref{linsys}. We now want to show that this latter
admits a unique solution in $C^{\frac{2+\alpha}{2},2
+{\alpha}}\left(\left[0,T\right]\times\left[0,1\right]\right)$.
This is due to
general results by Solonnikov~\cite{solonnikov1},
provided  the so--called 
\textbf{complementary conditions} hold (see~\cite[p.~11]{solonnikov1}).
The theory of Solonnikov is a 
generalization  to parabolic systems of the elliptic theory by Agmon, Douglis 
and Nirenberg.

\medskip

The complementary conditions are algebraic conditions 
that the matrices that represent the boundary operator and 
the initial datum have to satisfy
(see also~\cite[p.~97]{solonnikov1}).
Showing this conditions for a particular system
can be
heavy from the computational point of view.
For instance in~\cite[pages 11--15]{eidelman2} it is proved that 
the complementary condition follows from
the \textbf{Lopatinskii--Shapiro condition}.
We state here the definition of Lopatinskii--Shapiro condition at the triple junction, it is
similar at the end--points on $\partial\Omega$.

\begin{defn}
Let $\lambda\in\mathbb{C}$ with $ \Re(\lambda)>0$ be arbitrary.
The Lopatinskii--Shapiro condition for system~\eqref{linsys} is satisfied at the triple junction if 
every solution $(\gamma^i)_{i=1,2,3}\in C^2([0,\infty),(\mathbb{C}^2)^3)$ to
\begin{equation}\label{LopatinskiiShapirosystem}
\begin{cases}
\begin{array}{llll}
\lambda \gamma^i(x)-\frac{1}{\vert\sigma^i_x(0)
\vert^2}\gamma^i_{xx}(x)&=0&\;x\in[0,\infty),
i\in\{1,2,3\}&\;\text{motion,}\\
\gamma^{1}(0)-\gamma^{2}(0)&=0 &\; &\;\text{concurrency,}\\
\gamma^{2}(0)-\gamma^{3}(0)&=0 &\; &\;\text{concurrency,}\\
\sum_{i=1}^3\frac{\gamma_{x}^{i}(x)}{\left|\sigma_{x}^{i}(0)\right|} -\frac{\sigma^i_x(0)\left\langle \gamma^i_x(x),\sigma^i_x(0)\right\rangle}{\vert \sigma^i_x(0)\vert^3}
&=0 &\;&\;\text{angle condition,}
\end{array}
\end{cases}
\end{equation}
which satisfies $\lim_{x\to\infty}\lvert \gamma^i(x)\rvert=0$ is the trivial solution. 
\end{defn}

The angle condition in the previous system can be equivalently written
as 
$$\sum_{i=1}^3\frac{1}{\vert \sigma_x(0)^i\vert ^3}
\left\langle \gamma^i_{x}(x),\nu_0^i(0)\right\rangle \nu^i_0(0)=0\,.$$

It can be proved that 
Lopatinskii--Shapiro condition for system~\eqref{linsys} is satisfied 
testing the motion equation
by $\vert\sigma(0)^i_x\vert\overline{\langle\gamma^i(x),\nu^i(0)\rangle}\nu^i(0)$ and then 
 by $\vert\sigma(0)^i_x\vert\overline{\langle\gamma^i(x),\tau^i(0)\rangle}\tau^i(0)$
 and using the concurrency and the angle conditions.
 
 Once it is shown that the complementary conditions are fulfilled, 
 then~\cite[Theorem~4.9]{solonnikov1} guarantees
 existence and uniqueness of a solution of system~\eqref{linsys}.
 
\medskip
 
For $T>0$
we define the map $L_T:X_T\to Y_T$ as
$$
L(\gamma)=
\begin{pmatrix}
\left(\gamma^i_t-\frac{1}{\vert\sigma^i_x\vert^2}\gamma^i_{xx}\right)_{i\in\{1,2,3\}}\\
-\sum_{i=1}^3\frac{\gamma_{x}^{i}}{\left|\sigma_{x}^{i}\right|} -\frac{\sigma^i_x\left\langle \gamma^i_x,\sigma^i_x\right\rangle}{\vert \sigma^i_x\vert^3}\,\Big\vert _{x=0}\\
\gamma^i_{\vert x=1}\\
\gamma_{\vert t=0}
\end{pmatrix}
$$

where the linear spaces $X_T$ and  $Y_T$ are
\begin{align*}
X_T:=\{&\gamma\in 
C  ^{\frac{2+\alpha}{2},{^{2+\alpha}}}([0,T]\times[0,1];(\mathbb{R}^2)^3)
\;\text{such that for}\; t\in[0,T]\,,i\in\{1,2,3\}\\
&\text{it holds}\,
\gamma^1(t,1)=\gamma^2(t,2)=\gamma^3(t,3)\}
\,,\\
Y_T:=\{&(f,b,\psi)\in 
C  ^{\frac{\alpha}{4},{^{\alpha}}}([0,T]\times[0,1];(\mathbb{R}^2)^3)
\times 
C  ^{\frac{1+\alpha}{2}}([0,T];\mathbb{R}^4)
\times  C^{2+\alpha}\left([0,1];\left(\mathbb{R}^2\right)^3\right)\\
&\text{such that the linear compatibility conditions hold}
\}
\,,
\end{align*}
endowed with the induced norms. 
Then as a consequence of
the existence and uniqueness of a solution of system~\eqref{linsys}
we get that $L_T$ is a continuous isomorphism.

\begin{rem}
 The linearized version of 
  $\frac{\gamma_{xx}}{\vert\gamma_x\vert^2}$
(linearising around $\sigma$)
is 
\begin{align}\label{linearizzato}
\frac{1}{\vert\sigma_x\vert^2}\gamma_{xx}
-2\frac{\sigma_{xx}\left\langle \gamma_x,\sigma_x\right\rangle }{\vert \sigma_x\vert^4}\,.
\end{align}
As the well posedness of system~\eqref{linsys} depends only on the 
highest order term
we can restrict to consider~\eqref{linmotion} instead of~\eqref{linearizzato}.
\end{rem}

\medskip

\textit{Step 3: Fixed point argument} 
 
In the last step of the proof 
we deduce existence of a solution for system~\eqref{problema} 
from the linear problem by a  \textbf{contraction argument}.  

Let us define the operator $N$ that ``contains the information"
about the non--linearity of our problem.
The two components of this map are the following:

\begin{align*}
N_{1}:&
\begin{cases}
X^{\varphi,P}_T &\to C  ^{\frac\alpha2,
{^{\alpha}}}([0,T]\times[0,1];(\mathbb{R}^2)^3),\\
\gamma&\mapsto
f(\gamma),
\end{cases}\\
N_{2}:&
\begin{cases}
X^{\varphi,P}_T&\to  C^{\frac{1+\alpha}{2}}([0,T];\mathbb{R}^4), \\
\gamma&\mapsto b(\gamma)
\end{cases}
\end{align*}
where  $X^{\varphi,P}_T=\{\gamma\in 
X_T\,\text{such that }\,\gamma_{\vert t=0}=\varphi\;\text{and}\,
\gamma^i(t,1)=P^i\;\text{for}\;i\in\{1,2,3\}\}$.

Then $\gamma$ is a solution for system~\eqref{problema} if and only if 
$\gamma\in X^\varphi_T$ and 
\begin{equation*}
L_T(\gamma)=N_T(\gamma)\qquad \Longleftrightarrow
\qquad  \gamma=L_T^{-1}N_T(\gamma):= K_T(\gamma)\,.
\end{equation*}
Hence there exists a unique solution to  system~\eqref{problema}  
if and only if 
$K_T:X^{\varphi,P}_T\to X^{\varphi,P}_T$ has a unique fixed point. 
By the contraction mapping principle it is enough to show that $K$ is a contraction.
This result conclude the proof of Theorem~\ref{2smoothexist0-triod}.
\qed

\medskip

The method of  Bronsard and Reitich extends to the case of a networks
with several $3$--points and end--points. Indeed  such method 
relies on the uniform parabolicity of the system (which is the
same) and on the fact that the complementary and compatibility
conditions are satisfied.

We have only to define what is an admissible initial parametrization of
a network.

\begin{defn}[Admissible initial parametrization of a network]\label{2compcond}
We say that a 
parametrization $\varphi=(\varphi^1,\ldots,\varphi^n)$
of a geometric admissible network $\mathcal{N}_0$ 
composed by $n$ curves
(hence such that $\cup_{i=1}^n\varphi^i([0,1])=\mathcal{N}_0$)
is an admissible initial one if 
each curve $\varphi^i$ is regular and of class $C^{2+\alpha}([0,1])$,
at the end---points $\varphi^i(1)=P^i$ it holds
$\varphi_{xx}^i(1)=0$
and at any $3$--point $O^p$
we have 
\begin{align*}
&\varphi^{p1}(O^p)=\varphi^{p2}(O^p)=\varphi^{p3}(O^p)\,,\\
&\frac{\varphi_x^{p1}(O^p)}{\vert \varphi_x^{p1}(O^p)\vert}
+\frac{\varphi_x^{p2}(O^p)}{\vert \varphi_x^{p2}(O^p)\vert}+
\frac{\varphi_x^{p3}(O^p)}{\vert \varphi_x^{p3}(O^p)\vert}=0\,,\\
&\frac{\varphi_{xx}^{p1}(O^p)}{|\varphi_x^{p1}(O^p)|^2}=
\frac{\varphi_{xx}^{p2}(O^p)}{|\varphi_x^{p2}(O^p)|^2}=
\frac{\varphi_{xx}^{p3}(O^p)}{|\varphi_x^{p3}(O^p)|^2}
\end{align*}
where we abused a little the notation as in Definition~\ref{probdef}.
\end{defn}

\begin{thm}\label{2smoothexist0}   
Given an admissible initial 
parametrization $\varphi=(\varphi^1,\ldots,\varphi^n)$
of a geometric admissible network $\mathcal{N}_0$, there exists a unique solution
$\gamma=(\gamma^1,\ldots,\gamma^n)$
in
$C^{\frac{2+\alpha}{2},2+\alpha}\left(\left[0,T\right]\times\left[0,1\right]\right)$
of the following system
\begin{equation}\label{problema-nogauge-general}
\begin{cases}
\begin{array}{lll}
\gamma^i_t(t,x)=\frac{\gamma_{xx}^{i}\left(t,x\right)}{\left|\gamma_{x}^{i}\left(t,x\right)\right|^{2}}
\qquad &&\text{ motion by curvature}\\
\gamma^{pj}\left(t,O^p\right)=\gamma^{pk}\left(t,O^p\right)\quad&\text{at every $3$--point $O^p$}
\quad&\text{ concurrency}\\
\sum_{j=1}^{3}\frac{\gamma_{x}^{pj}\left(t,O^p\right)}{\left|\gamma_{x}^{pj}\left(t,O^p\right)\right|}=0\quad&\text{at every $3$--point $O^p$}
\quad&\text{ angles condition}\\
\gamma^r(t,1)=P^r\quad&\text{with $0\leq r \leq l$}\,
\quad &\text{ Dirichlet boundary condition}\\
\gamma^i(x,0)=\varphi^i(x)
\qquad &&\text{ initial data}\\
\end{array}
\end{cases}
\end{equation}
(where we used the notation of Definition~\ref{probdef}) for every $x\in[0,1]$, $t\in[0,T]$ and
$i\in\{1,2,\dots, n\}$, $j\neq k\in\{1,2,3\}$ in a positive time interval $[0,T]$.
\end{thm}

\subsection{Existence and uniqueness}

In the previous section we have explained how to obtain a unique solution
for short time  to system~\eqref{problema} 
and more in general to system~\eqref{problema-nogauge-general}, 
but till now we have not solved our original problem yet.
Indeed in Definition~\ref{geosolution}
of solution of the motion by curvature 
appears a slightly different system.
Moreover Theorem~\ref{2smoothexist0-triod} (and Theorem~\ref{2smoothexist0})    
provides a solution given an \textbf{admissible initial parametrization}
but in Definition~\ref{geosolution} we speak of  \textbf{geometrically admissible initial network}.
It is then clear that we have to establish a relation between this two notions.

To this aim the following lemma will be useful.
\begin{lem}\label{propatjunction}
Consider a triple junction $O$ where the curves $\gamma^1,\gamma^2$ and $\gamma^3$
concur forming angles of $120$ degrees
(that is $\sum_{i=1}^3\tau^i=\sum_{i=1}^3\nu^i=0$).
Then
\begin{equation*}
k^1\nu^1+\lambda^1\tau^1
=k^2\nu^2+\lambda^2\tau^2
=k^3\nu^3+\lambda^3\tau^3\,,
\end{equation*}
is satisfied if and only if 
\begin{equation*}
k^1+k^2+k^3=0\qquad\text{and}\qquad
\lambda^1+\lambda^2+\lambda^3=0\,.
\end{equation*}
\end{lem}
\begin{proof}
Suppose that for $i\neq j\in\{1,2,3\}$ we have
$$
k^{i}\nu^{i}+\lambda^{i}\tau^{i}=
k^{j}\nu^{j}+\lambda^{j}\tau^{j}\,.
$$
Multiplying these vector equalities
by $\tau^{l}$ and $\nu^{l}$ and varying $i, j, l$,  thanks to the
conditions $\sum_{i=1}^{3}\tau^{pi}=\sum_{i=1}^{3}\nu^{pi}=0$, we get
the relations
\begin{gather*}
\lambda^{i}=-\lambda^{i+1}/2-\sqrt{3}k^{i+1}/2\\
\lambda^{i}=-\lambda^{i-1}/2+\sqrt{3}k^{i-1}/2\\
k^{i}=-k^{i+1}/2+\sqrt{3}\lambda^{i+1}/2\\
k^{i}=-k^{i-1}/2-\sqrt{3}\lambda^{i-1}/2
\end{gather*}
with the convention that the second superscripts are to be 
considered ``modulus $3$''. Solving this system we get
\begin{gather*}
\lambda^{i}=\frac{k^{i-1}-k^{i+1}}{\sqrt{3}}\\
k^{i}=\frac{\lambda^{i+1}-\lambda^{i-1}}{\sqrt{3}}
\end{gather*}
which implies
\begin{equation}\label{eq:cond2} 
\sum_{i=1}^3 k^{i}=\sum_{i=1}^3\lambda^{i}=0\,.
\end{equation}
\end{proof}

It is also possible to prove that at each triple junction the following properties hold
\begin{align}
&\sum_{i=1}^3 (k^{i})^2=\sum_{i=1}^3 (\lambda^{i})^2
\qquad\text{ { and}   }\qquad \sum_{i=1}^3 k^{i}\lambda^{pi}=0\,, \nonumber\\
&\dert^l\sum_{i=1}^3 k^{pi}=\sum_{i=1}^3\dert^l k^{pi}
=\dert^l\sum_{i=1}^3  \lambda^{pi}=\sum_{i=1}^3\dert^l\lambda^{pi}
=\dert\sum_{i=1}^3k^{pi}\lambda^{pi}=0\,,,\label{lambdakappa1}\\
&\sum_{i=1}^3 (\dert^l k^{pi})^2=\sum_{i=1}^3 (\dert^l\lambda^{pi})^2\,\,
\text{ { for every} $l\in\NN$,}\,\label{lambdakappa2}\\
&\dert^m(k^{pi}_s+\lambda^{pi} k^{pi})=
\dert^m(k^{pj}_s+\lambda^{pj} k^{pj}) \,\,\text {  for every pair $i, j$ and $m\in\NN$.}
,\label{lambdakappa3}\\
 &\sum_{i=1}^3 \dert^lk^{pi}\,\dert^m(k^{pi}_s+\lambda^{pi} k^{pi}) 
=\sum_{i=1}^3 \dert^l\lambda^{pi}\,\dert^m(k^{pi}_s+\lambda^{pi} k^{pi})=0 \,\,
\text { for  every $l, m\in\NN$.}\label{eq:orto}\\ 
\end{align}

We are ready now to establish the relation between geometrically 
admissible initial networks and admissible parametrizations.

\begin{lem}\label{repara}
Suppose that $\mathbb{T}_0$ is a geometrically admissible initial triod parametrized by 
$\gamma=(\gamma^1,\gamma^2,\gamma^3)$. Then there exist
three smooth functions $\theta^i:[0,1]\to [0,1]$ such that the reparametrization 
$\varphi:=\left(\gamma^1\circ\theta^1,\gamma^2\circ\theta^2,\gamma^3\circ\theta^3\right)$
is an admissible initial parametrization.
\end{lem}

\begin{proof}
Consider $\gamma=(\gamma^1,\gamma^2,\gamma^3)$ the parametrization 
of class $C^{2+\alpha}$ of 
$\mathbb{T}_0$ (that exists as $\mathbb{T}_0$ is a geometrically 
admissible initial triod).
It is not restrictive to suppose that 
$\gamma^1(0)=\gamma^2(0)=\gamma^3(0)$ is the triple junction and that
$\gamma^i(1)=P^i\in\partial\Omega$ with $i\in\{1,2,3\}$.

\smallskip

We look for smooth maps $\theta^i:[0,1]\to [0,1]$ such that 
 $\theta^i_x(x)\neq 0$
for every $x\in [0,1]$,  $\theta^i(0)=0$ and $\theta^i(1)=1$.
Then conditions 1. 2. 3. and 6. of Definition~\ref{admissible} are satisfied.

\smallskip

Condition 4. at the triple junction is true for any choice of the $\theta^i$
as it involves the unit tangent vectors 
that are invariant under reparametrization.

\smallskip

We pass now to Condition 5. namely we want that 
\begin{equation}\label{cond5}
\frac{\varphi_{xx}^1}{|\varphi_x^1|^2}
=\frac{\varphi_{xx}^2}{|\varphi_x^2|^2}
=\frac{\varphi_{xx}^3}{|\varphi_x^3|^2}
\,.
\end{equation}

We indicate with the subscript $\gamma$ or $\varphi$ the geometric quantities
computed for the parametrization $\gamma$ or $\varphi$, respectively.
We define  $\lambda^i:=\frac{\langle\varphi_{xx}^i\,\vert\varphi^i_x\rangle}{|\varphi_x^i|^3}$.
Then~\eqref{cond5} can be equivalently written as 
\begin{equation}\label{cond5v1}
k^1_\varphi\nu^1_\varphi+\lambda^1\tau^1_\varphi
=k^2_\varphi\nu^2_\varphi+\lambda^2\tau^2_\varphi
=k^3_\varphi\nu^3_\varphi+\lambda^3\tau^3_\varphi\,,
\end{equation}
and, as all the geometric quantities involved are invariant under reparametrization,
the equality~\eqref{cond5v1}
is nothing else than  
\begin{equation}
k^1_\gamma\nu^1_\gamma+\lambda^1\tau^1_\gamma
=k^2_\gamma\nu^2_\gamma+\lambda^2\tau^2_\gamma
=k^3_\gamma\nu^3_\gamma+\lambda^3\tau^3_\gamma\,,
\end{equation}
that by Lemma~\ref{propatjunction} is satisfied if and only if 
\begin{equation}\label{relation1}
k^1_\gamma+k^2_\gamma+k^3_\gamma=0\qquad\text{and}\qquad
\lambda^1+\lambda^2+\lambda^3=0\,.
\end{equation}

\smallskip

To satisfy Condition 7. we need a similar request.
Indeed  ${\varphi_{xx}^i}=0$ at every end--point of the network
is equivalent to the condition $k_\gamma^i\nu_\gamma^i+{\lambda}^i\tau_\gamma^i=0$,
that is satisfied if and only if 
\begin{equation}\label{relation2}
k^i_\gamma=0\qquad\text{and}\qquad \lambda^i=0
\end{equation}
at every end--point of the network.

\smallskip

Hence, we only need to find $C^\infty$ reparametrizations $\theta^i$ such that at the borders of $[0,1]$ the values of $\lambda^i$ are given by the relations in~\eqref{relation1} and~\eqref{relation2}.
 This can be easily done since at the borders of the interval $[0,1]$ we have $\theta^i(0)=0$ and $\theta^i(1)=1$, hence 
$$
\lambda^i=\frac{\langle\varphi_{xx}^i\,\vert\varphi^i_x\rangle}{|\varphi_x^i|^3}=-\partial_x\frac{1}{\vert\varphi_x^i\vert}
=-\partial_x\frac{1}{\vert\gamma_x^i\comp\,\theta^i\vert\theta_x^i}
=\frac{\langle{\gamma}_{xx}^i\,\vert{\gamma}^i_x\rangle}{|{\gamma}_x^i|^3}+\frac{\theta_{xx}^i}{\vert\sigma_x^i\vert\vert\theta_x^i\vert^2}
=\lambda_\gamma^i+\frac{\theta_{xx}^i}{\vert\sigma_x^i\vert\vert\theta_x^i\vert^2}
$$ 
where $\lambda_\gamma^i=\frac{\langle{\gamma}_{xx}^i\,\vert{\gamma}^i_x\rangle}{|{\gamma}_x^i|^3}$.

\smallskip

Choosing any $C^\infty$ functions $\theta^i$ with $\theta_x^i(0)=\theta_x^i(1)=1$, $\theta(1)_{xx}^i=-\lambda_\gamma^i|\gamma_x^i||\theta_x^i|^2$  and 
$$
\theta(0)_{xx}^i=\left(\frac{k_\gamma^{i-1}-k_\gamma^{i+1}}{\sqrt{3}}-\lambda_\gamma^i\right)\,\vert\gamma_x^i\vert\vert\theta_x^i\vert^2
$$
(for instance, one can use a polynomial function)
the reparametrization $\varphi=(\varphi^1,\varphi^2,\varphi^3)$ satisfies Conditions 1. to 7.
of Definition~\ref{admissible} and
the proof is completed.
\end{proof}

\begin{rem}
Vice versa if $\varphi$ is an admissible initial parametrization,
then the triod $\cup_{i=^1}^3\varphi^i([0,1])$ is clearly a geometrically admissible 
initial network.
Indeed one uses Lemma~\ref{propatjunction} 
to get that the sum of the curvature at the junction is zero.
The other properties are trivially verified.
\end{rem}

We are ready now to discuss existence and uniqueness of solution of the geometric problem.
We need to introduce the notion of geometric uniqueness
because
even if the solution $\gamma$ of system~\eqref{problema} is unique,
there are anyway several solutions of Problem~\ref{probgradientflow}
obtained by  reparametrizing $\gamma$.

\begin{defn}\label{uniqdef}
We say that Problem~\ref{probgradientflow}
admits a  \textbf{geometrically unique} solution if 
there exists a unique family of time--dependent networks (sets) $(\mathcal{N}_t)_{t\in[0,T]}$
satisfying the definition of solution~\ref{geosolution}.
\end{defn}
In particular this means that all the solutions (functions) 
satisfying system~\eqref{problemageometrico} 
can be obtained one from each other by means of time--depending reparametrization.

\begin{thm}[Geometric uniqueness]\label{geouniq}
Let $\mathbb{T}_0$ be a geometrically admissible initial triod. Then
there exists a geometrically unique solution of Problem~\eqref{problema}
in a positive time interval $[0,\widetilde{T}]$.
\end{thm}
\begin{proof}
Let $\mathbb{T}_0$ be a geometrically admissible initial triod parametrized by 
$\gamma_0=(\gamma_0^1,\gamma_0^2,\gamma_0^3)$
admissible initial parametrization (that always exists thanks to Lemma~\ref{repara}). 
Then by Theorem~\ref{2smoothexist0-triod}   there exists a unique solution 
$\gamma=(\gamma^1,\gamma^2,\gamma^3)$
to system~\eqref{problema} with initial data
$\gamma_0=(\gamma_0^1,\gamma_0^2,\gamma_0^3)$ in a positive time interval
$[0,T]$.
In particular 
$(\mathbb{T}_t)_{t\in[0,T]}=(\cup_{i=1}^3\gamma^i([0,1])_t)_{t\in[0,T]}$
is a solution of the motion by curvature.

\medskip

Suppose by contradiction that there exists another solution
$(\widetilde{\mathbb{T}}_t)_{t\in[0,T']}$ to Problem~\ref{probgradientflow}
with the same initial $\mathbb{T}_0$ .
Let this solution be
parametrized by $\tilde{\gamma}=(\tilde{\gamma}^1,\tilde{\gamma}^2,\tilde{\gamma}^3)$
with 
$$
\widetilde{\gamma}^i\in C^{\frac{2+\alpha}{2},2+\alpha}([0,T']\times[0,1])\,.
$$

We want to show that the sets $\mathbb{T}$ and $\widetilde{\mathbb{T}}$ coincide, namely that 
$\tilde{\gamma}$ coincides to $\gamma$ up to a
reparametrization of the curves $\widetilde{\gamma}(\cdot,t)$ for
every $t\in[0,\min\{T,T'\})$.

Let $\varphi^i:[0,\min\left\{ T,T'\right\} ]\times[0,1]\to[0,1]$ be
in  $C^{\frac{2+\alpha}{2},2+\alpha}([0,\min\{T,T'\}]\times[0,1])$ and 
consider the
reparametrizations $\overline{\gamma}^i(t,x)=\widetilde{\gamma}^i(t,\varphi^i(t,x))$.
 We have 
 $\overline{\gamma}^i\in C^{\frac{2+\alpha}{2},2+\alpha}([0,\min\{T,T'\})\times[0,1])$ and 
\begin{align*}
\overline{\gamma}^i_t(t,x)
=&\,\partial_t[\widetilde{\gamma}^i(t,\varphi^i(t,x))]\\
=&\, \widetilde{\gamma}^i_t(t,\varphi^i(t,x)) 
+\widetilde{\gamma}^i_x(t,\varphi^i(t,x))\varphi^i_t(t,x)\\
=&\,
\widetilde{k}^i(t,\varphi^i(t,x))\widetilde{\nu}^i(t,\varphi^i(t,x))
+\widetilde{\lambda}^i(t,\varphi^i(t,x))\widetilde{\tau}^i(t,\varphi^i(t,x))\\
+&\,\widetilde{\gamma}^i_x(t,\varphi^i(t,x))\varphi^i_t(t,x)\\
=&\,\left\langle\frac{\widetilde{\gamma}_{xx}^{i}\left(t,\varphi^i(t,x)\right)}
{\left|\widetilde{\gamma}_{x}^{i}\left(t,\varphi^i(t,x)\right)\right|^{2}}\,
\Big\vert\,\widetilde{\nu}^i(t,\varphi^i(t,x))\right\rangle
\widetilde{\nu}^i(t,\varphi^i(t,x))\\
+&\,\widetilde{\lambda}^i(t,\varphi^i(t,x))
\frac{\widetilde{\gamma}_x^i(t,\varphi^i(t,x))}{\left|\widetilde{\gamma}_{x}^{i}\left(t,
\varphi^i(t,x)\right)\right|}
+\,\widetilde{\gamma}^i_x(t,\varphi^i(t,x))\varphi^i_t(t,x)\,.\\
\end{align*}
We ask now the maps $\varphi^i$
to be
solutions for some positive interval of time $[0,T'']$ of the following quasilinear PDE's 
\begin{align}\label{reparphi}
\varphi^i_t(t,x)
=&\frac{1}{\left|\widetilde{\gamma}_{x}^{i}\left(t,
\varphi^i(t,x)\right)\right|}
\left\langle\frac{\widetilde{\gamma}_{xx}^{i}\left(t,\varphi^i(t,x)\right)}{\left|
\widetilde{\gamma}_{x}^{i}\left(t,\varphi^i(t,x)\right)\right|^{2}}\,\Big\vert\,
\frac{\widetilde{\gamma}_x^i(t,\varphi^i(t,x))}{\left|\widetilde{\gamma}_{x}^{i}\left(t,
\varphi^i(t,x)\right)\right|}\right\rangle\nonumber \\
-&\frac{\widetilde{\lambda}^i(t,\varphi^i(t,x))}{\left|\widetilde{\gamma}_{x}^{i}\left(t,
\varphi^i(t,x)\right)\right|}
+\frac{\varphi^i_{xx}(t,x)}{\left|\widetilde{\gamma}_{x}^{i}\left(t,\varphi^i(t,x)\right)\right|^2
\left|\varphi^i_x(t,x)\right|^2}\,,
\end{align}
with $\varphi^i(t,0)=0$, $\varphi^i(t,1)=1$,  $\varphi^i(0,x)=x$ (hence, 
$\overline{\gamma}^i(0,x)=\gamma^i(0,x)=\sigma^i(x)$) and $\varphi_x(t,x)\not=0$ . 
The existence of such solutions follows by standard theory of second order quasilinear parabolic 
equations (see~\cite{lasolura,lunardi1}).
Then we have
\begin{align*}
\overline{\gamma}^i_t(t,x)=
&\,\left\langle\frac{\widetilde{\gamma}_{xx}^{i}\left(t,\varphi^i(t,x)\right)}
{\left|\widetilde{\gamma}_{x}^{i}\left(t,\varphi^i(t,x)\right)\right|^{2}}\,
\Big\vert\,\widetilde{\nu}^i(t,\varphi^i(t,x))\right\rangle
\widetilde{\nu}^i(t,\varphi^i(t,x))\\
+&\,\left\langle\frac{\widetilde{\gamma}_{xx}^{i}\left(t,\varphi^i(t,x)\right)}{\left|
\widetilde{\gamma}_{x}^{i}\left(t,\varphi^i(t,x)\right)\right|^{2}}\,\Big\vert\,
\frac{\widetilde{\gamma}_x^i(t,\varphi^i(t,x))}{\left|\widetilde{\gamma}_{x}^{i}\left(t,
\varphi^i(t,x)\right)\right|}\right\rangle
\frac{\widetilde{\gamma}_{x}^{i}\left(t,
\varphi^i(t,x)\right)}{\left|\widetilde{\gamma}_{x}^{i}\left(t,
\varphi^i(t,x)\right)\right|}
\nonumber \\
+&\,\frac{\varphi^i_{xx}(t,x)\widetilde{\gamma}_{x}^{i}\left(t,\varphi^i(t,x)\right)}
{\left|\widetilde{\gamma}_{x}^{i}\left(t,\varphi^i(t,x)\right)\right|^2
\left|\varphi^i_x(t,x)\right|^2}\\
=&\,\frac{\widetilde{\gamma}_{xx}^{i}\left(t,\varphi^i(t,x)\right)}
{\left|\widetilde{\gamma}_{x}^{i}\left(t,\varphi^i(t,x)\right)\right|^{2}}
+\frac{\varphi^i_{xx}(t,x)\widetilde{\gamma}_{x}^{i}\left(t,\varphi^i(t,x)\right)}
{\left|\widetilde{\gamma}_{x}^{i}\left(t,\varphi^i(t,x)\right)\right|^2
\left|\varphi^i_x(t,x)\right|^2}\\
=&\,\frac{\overline{\gamma}_{xx}^{i}(t,x)}
{\vert\overline{\gamma}_{x}^{i}(t,x)|^{2}}\,.
\end{align*}
By the uniqueness result of Theorem~\ref{2smoothexist0}
we can then conclude that
$\overline{\gamma}^i=\gamma^i$ for every $i\in\{1,2,\dots, n\}$, 
hence ${\gamma}^i(t,x)=\widetilde{\gamma}^i(t,\varphi^i(t,x))$ in the time interval 
$[0,\widetilde{T}]$ where $\widetilde{T}:=\min\{T,T', T''\}$.
\end{proof}

\subsection{Geometric properties of the flow}
\label{geopropsub}

In Definition~\ref{network} of network we require that the curves are injective and regular.
The second assumption is needed to define the flow because $\vert \gamma_x\vert$ 
appears at the denominator.
For the short time existence of the flow we did not require that the curves are embedded.
We now show that if the initial network is embedded then 
 the evolving networks stay embedded and intersect the boundary of $\Omega$ only at the fixed end--points (transversally).

\begin{prop}\label{omegaok} 
Let $\mathcal{N}_t$ be the curvature flow of a regular network in a smooth, convex,
bounded, open set $\Omega$, with fixed end--points on the boundary of
$\Omega$, for $t\in[0,T)$. Then, for every time $t\in[0,T)$, the
network $\mathcal{N}_t$ intersects the boundary of $\Omega$ only at the
end--points and such intersections are transversal for every positive
time. Moreover, $\mathcal{N}_t$ remains embedded.
\end{prop}

\begin{proof}
By continuity, the $3$--points cannot hit the boundary of $\Omega$ at least 
for some time $T^\prime>0$.  The convexity of $\Omega$ and 
the strong maximum principle (see~\cite{prowein}) 
imply that the network cannot intersect the boundary for the first time
at an inner regular point. 
As a consequence, if $t_0>0$ is the ``first time'' when
the $\mathcal{N}_t$ intersects the boundary at an inner point, this latter has
to be a $3$--point. The minimality of $t_0$ is then easily
contradicted by the convexity of $\Omega$, the $120$ degrees condition
and the nonzero length of the curves of $\mathcal{N}_{t_0}$.\\
Even if some of the curves of the initial network are tangent to
$\partial\Omega$ at the end--points, by the strong maximum
principle, as $\Omega$ is convex, the intersections become immediately
transversal and stay so for every subsequent time.\\
Finally, if the evolution $\mathcal{N}_t$ loses embeddedness for the first time,
this cannot happen neither 
at a boundary point, by the argument above, nor  
at a $3$--point, by the $120$ degrees condition.
Hence it must happen at interior regular points, 
but this contradicts the strong maximum principle.
\end{proof}

\begin{prop}\label{omegaok2} 
In the same hypotheses of the previous proposition, if the smooth, bounded, open set $\Omega$ is strictly convex, for every fixed end--point $P^r$ on the boundary of $\Omega$, for $r\in\{1,2,\dots,l\}$, there is a time $t_r\in(0,T)$ and an angle $\alpha_r$ smaller than $\pi/2$ such that the curve of the network arriving at $P^r$ form an angle less that $\alpha_r$ with the inner normal to the boundary of $\Omega$, for every time $t\in(t_r,T)$.
\end{prop}
\begin{proof}
We observe that the evolving network $\mathcal{N}_t$ is contained in the convex set $\Om_t\subset \Om$,
obtained by letting $\partial \Om$ (which is a finite set of smooth curves with end--points $P^r$) move by curvature 
keeping fixed the end--points $P^r$ (see~\cite{huisk2,stahl1,stahl2}). By the strict convexity of $\Omega$ and strong maximum principle, for every positive $t>0$, the two curves of the boundary of $\Omega$ concurring at $P^r$ form an angle smaller that $\pi$ which is not increasing in time. Hence, the statement of the proposition follows.
\end{proof}

\section{Self-similar solutions}\label{selfsimilar}

Once established the existence of solution for a short time, we want to analyse
the behavior of the flow in the long time.
A good way to understand more about the flow is looking for examples of solutions.

A straight line is perhaps the easiest example. It is also easy to see that
an infinite flat triod with the triple junction at the origin (called \textbf{standard triod}) is a solution (Figure~\ref{standard}).

\begin{figure}[H]
\begin{center}
\begin{tikzpicture}[scale=0.5]
\draw[shift={(-7.5,0)}]
(-2,-2) to [out=45, in=-135, looseness=1] (2,2);
\draw[dashed,shift={(-7.5,0)}]
(-2.5,-2.5) to [out=45, in=-135, looseness=1] (-2,-2)
(2,2)to [out=45, in=-135, looseness=1](2.6,2.6);
\draw
(0,0) to [out=90, in=-90, looseness=1] (0,2)
(0,0) to [out=210, in=30, looseness=1] (-1.73,-1)
(0,0) to [out=-30, in=150, looseness=1](1.73,-1);
\draw[dashed]
(0,2) to [out=90, in=-90, looseness=1] (0,3)
(-1.73,-1) to [out=210, in=30, looseness=1] (-2.59,-1.5)
(1.73,-1) to [out=-30, in=150, looseness=1](2.59,-1.5);
\fill(0,0) circle (2pt);
\fill(-7.5,0) circle (2pt);
\path[font=\small]
(-7.5,.2) node[left]{$O$}
(.1,.2) node[left]{$O$};
\end{tikzpicture}
\caption{A straight line and a standard triod are solutions of the motion by curvature}\label{standard}
\end{center}
\end{figure}
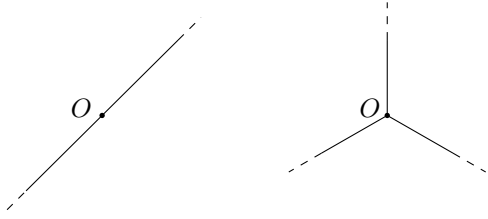

In both these examples the existence is global in time and the set does not change shape 
during the evolution.
From this last observation one could guess that there is an entire class of
solutions that preserve their shape in time.
We try to classify now
self--similar solution in a 
systematic way.

\medskip

Let us start looking for self--similar
\textbf{translating} solutions.

Suppose that we 
have a translating curve $\gamma$ solving the motion by curvature
with initial data $\sigma$. We
can write $\gamma(t,x)=\eta(x)+w(t)$.
The motion by curvature equation $k(t,x)=\left\langle\gamma_t(t,x),\nu(t,x)\right\rangle$
in this case reads as
$k(x)
=\left\langle w'(t),\nu(x)\right\rangle$.
As a consequence $w(t)$ is constant, hence we are allowed to write
$\gamma(t,x)=\eta(x)+t\boldsymbol{v}$ with $\boldsymbol{v}\in\mathbb{R}^2$,
and we obtain 
$$
k(x)=\left\langle\boldsymbol{v},\nu(x)\right\rangle\,.
$$
The reverse is also true: if a curve $\gamma$ satisfies
$k(x)=\left\langle\boldsymbol{v},\nu(x)\right\rangle$,
then $\gamma$ is a translating solution of the curvature flow.
By integrating this ODE (with $\boldsymbol{v}=\boldsymbol{e}_1$) one can see that the only translating curve
is given by the graph of the function $x=-\log\cos y$ in the interval
$(-\frac{\pi}{2}, \frac{\pi}{2})$. Grayson in~\cite{gray1}
named this curve the \textbf{grim reaper} (Figure~\ref{figura2}).


\begin{figure}[h]
\setlength{\unitlength}{0.240900pt}
\ifx\plotpoint\undefined\newsavebox{\plotpoint}\fi
\begin{picture}(1500,450)(0,0)
\font\gnuplot=cmr10 at 10pt
\gnuplot
\sbox{\plotpoint}{\rule[-0.200pt]{0.400pt}{0.400pt}}%
\put(451,249){\makebox(0,0)[l]{$e_1$}}
\put(1324,366){\makebox(0,0)[l]{$y={\pi}/2$}}
\put(1324,87){\makebox(0,0)[l]{$y=-{\pi}/2$}}
\put(405,225){\vector(1,0){184}}
\put(1216,41){\usebox{\plotpoint}}
\multiput(1085.24,41.60)(-45.810,0.468){5}{\rule{31.500pt}{0.113pt}}
\multiput(1150.62,40.17)(-248.620,4.000){2}{\rule{15.750pt}{0.400pt}}
\multiput(848.45,45.60)(-18.613,0.468){5}{\rule{12.900pt}{0.113pt}}
\multiput(875.23,44.17)(-101.225,4.000){2}{\rule{6.450pt}{0.400pt}}
\multiput(728.75,49.61)(-17.877,0.447){3}{\rule{10.900pt}{0.108pt}}
\multiput(751.38,48.17)(-58.377,3.000){2}{\rule{5.450pt}{0.400pt}}
\multiput(667.68,52.60)(-8.670,0.468){5}{\rule{6.100pt}{0.113pt}}
\multiput(680.34,51.17)(-47.339,4.000){2}{\rule{3.050pt}{0.400pt}}
\multiput(612.66,56.60)(-6.915,0.468){5}{\rule{4.900pt}{0.113pt}}
\multiput(622.83,55.17)(-37.830,4.000){2}{\rule{2.450pt}{0.400pt}}
\multiput(568.40,60.60)(-5.599,0.468){5}{\rule{4.000pt}{0.113pt}}
\multiput(576.70,59.17)(-30.698,4.000){2}{\rule{2.000pt}{0.400pt}}
\multiput(527.32,64.61)(-7.160,0.447){3}{\rule{4.500pt}{0.108pt}}
\multiput(536.66,63.17)(-23.660,3.000){2}{\rule{2.250pt}{0.400pt}}
\multiput(500.96,67.60)(-3.990,0.468){5}{\rule{2.900pt}{0.113pt}}
\multiput(506.98,66.17)(-21.981,4.000){2}{\rule{1.450pt}{0.400pt}}
\multiput(473.79,71.60)(-3.698,0.468){5}{\rule{2.700pt}{0.113pt}}
\multiput(479.40,70.17)(-20.396,4.000){2}{\rule{1.350pt}{0.400pt}}
\multiput(446.41,75.61)(-4.704,0.447){3}{\rule{3.033pt}{0.108pt}}
\multiput(452.70,74.17)(-15.704,3.000){2}{\rule{1.517pt}{0.400pt}}
\multiput(427.87,78.60)(-2.967,0.468){5}{\rule{2.200pt}{0.113pt}}
\multiput(432.43,77.17)(-16.434,4.000){2}{\rule{1.100pt}{0.400pt}}
\multiput(408.11,82.60)(-2.528,0.468){5}{\rule{1.900pt}{0.113pt}}
\multiput(412.06,81.17)(-14.056,4.000){2}{\rule{0.950pt}{0.400pt}}
\multiput(390.53,86.60)(-2.382,0.468){5}{\rule{1.800pt}{0.113pt}}
\multiput(394.26,85.17)(-13.264,4.000){2}{\rule{0.900pt}{0.400pt}}
\multiput(371.73,90.61)(-3.365,0.447){3}{\rule{2.233pt}{0.108pt}}
\multiput(376.36,89.17)(-11.365,3.000){2}{\rule{1.117pt}{0.400pt}}
\multiput(358.77,93.60)(-1.943,0.468){5}{\rule{1.500pt}{0.113pt}}
\multiput(361.89,92.17)(-10.887,4.000){2}{\rule{0.750pt}{0.400pt}}
\multiput(345.19,97.60)(-1.797,0.468){5}{\rule{1.400pt}{0.113pt}}
\multiput(348.09,96.17)(-10.094,4.000){2}{\rule{0.700pt}{0.400pt}}
\multiput(330.39,101.61)(-2.695,0.447){3}{\rule{1.833pt}{0.108pt}}
\multiput(334.19,100.17)(-9.195,3.000){2}{\rule{0.917pt}{0.400pt}}
\multiput(320.02,104.60)(-1.505,0.468){5}{\rule{1.200pt}{0.113pt}}
\multiput(322.51,103.17)(-8.509,4.000){2}{\rule{0.600pt}{0.400pt}}
\multiput(309.02,108.60)(-1.505,0.468){5}{\rule{1.200pt}{0.113pt}}
\multiput(311.51,107.17)(-8.509,4.000){2}{\rule{0.600pt}{0.400pt}}
\multiput(298.43,112.60)(-1.358,0.468){5}{\rule{1.100pt}{0.113pt}}
\multiput(300.72,111.17)(-7.717,4.000){2}{\rule{0.550pt}{0.400pt}}
\multiput(287.60,116.61)(-1.802,0.447){3}{\rule{1.300pt}{0.108pt}}
\multiput(290.30,115.17)(-6.302,3.000){2}{\rule{0.650pt}{0.400pt}}
\multiput(279.85,119.60)(-1.212,0.468){5}{\rule{1.000pt}{0.113pt}}
\multiput(281.92,118.17)(-6.924,4.000){2}{\rule{0.500pt}{0.400pt}}
\multiput(271.26,123.60)(-1.066,0.468){5}{\rule{0.900pt}{0.113pt}}
\multiput(273.13,122.17)(-6.132,4.000){2}{\rule{0.450pt}{0.400pt}}
\multiput(262.16,127.61)(-1.579,0.447){3}{\rule{1.167pt}{0.108pt}}
\multiput(264.58,126.17)(-5.579,3.000){2}{\rule{0.583pt}{0.400pt}}
\multiput(255.68,130.60)(-0.920,0.468){5}{\rule{0.800pt}{0.113pt}}
\multiput(257.34,129.17)(-5.340,4.000){2}{\rule{0.400pt}{0.400pt}}
\multiput(248.68,134.60)(-0.920,0.468){5}{\rule{0.800pt}{0.113pt}}
\multiput(250.34,133.17)(-5.340,4.000){2}{\rule{0.400pt}{0.400pt}}
\multiput(240.71,138.61)(-1.355,0.447){3}{\rule{1.033pt}{0.108pt}}
\multiput(242.86,137.17)(-4.855,3.000){2}{\rule{0.517pt}{0.400pt}}
\multiput(235.09,141.60)(-0.774,0.468){5}{\rule{0.700pt}{0.113pt}}
\multiput(236.55,140.17)(-4.547,4.000){2}{\rule{0.350pt}{0.400pt}}
\multiput(229.51,145.60)(-0.627,0.468){5}{\rule{0.600pt}{0.113pt}}
\multiput(230.75,144.17)(-3.755,4.000){2}{\rule{0.300pt}{0.400pt}}
\multiput(224.09,149.60)(-0.774,0.468){5}{\rule{0.700pt}{0.113pt}}
\multiput(225.55,148.17)(-4.547,4.000){2}{\rule{0.350pt}{0.400pt}}
\multiput(217.82,153.61)(-0.909,0.447){3}{\rule{0.767pt}{0.108pt}}
\multiput(219.41,152.17)(-3.409,3.000){2}{\rule{0.383pt}{0.400pt}}
\multiput(213.92,156.60)(-0.481,0.468){5}{\rule{0.500pt}{0.113pt}}
\multiput(214.96,155.17)(-2.962,4.000){2}{\rule{0.250pt}{0.400pt}}
\multiput(209.92,160.60)(-0.481,0.468){5}{\rule{0.500pt}{0.113pt}}
\multiput(210.96,159.17)(-2.962,4.000){2}{\rule{0.250pt}{0.400pt}}
\multiput(205.37,164.61)(-0.685,0.447){3}{\rule{0.633pt}{0.108pt}}
\multiput(206.69,163.17)(-2.685,3.000){2}{\rule{0.317pt}{0.400pt}}
\multiput(201.92,167.60)(-0.481,0.468){5}{\rule{0.500pt}{0.113pt}}
\multiput(202.96,166.17)(-2.962,4.000){2}{\rule{0.250pt}{0.400pt}}
\multiput(197.92,171.60)(-0.481,0.468){5}{\rule{0.500pt}{0.113pt}}
\multiput(198.96,170.17)(-2.962,4.000){2}{\rule{0.250pt}{0.400pt}}
\multiput(194.95,175.00)(-0.447,0.685){3}{\rule{0.108pt}{0.633pt}}
\multiput(195.17,175.00)(-3.000,2.685){2}{\rule{0.400pt}{0.317pt}}
\multiput(190.92,179.61)(-0.462,0.447){3}{\rule{0.500pt}{0.108pt}}
\multiput(191.96,178.17)(-1.962,3.000){2}{\rule{0.250pt}{0.400pt}}
\put(188.17,182){\rule{0.400pt}{0.900pt}}
\multiput(189.17,182.00)(-2.000,2.132){2}{\rule{0.400pt}{0.450pt}}
\multiput(186.95,186.00)(-0.447,0.685){3}{\rule{0.108pt}{0.633pt}}
\multiput(187.17,186.00)(-3.000,2.685){2}{\rule{0.400pt}{0.317pt}}
\put(183.17,190){\rule{0.400pt}{0.700pt}}
\multiput(184.17,190.00)(-2.000,1.547){2}{\rule{0.400pt}{0.350pt}}
\put(181.17,193){\rule{0.400pt}{0.900pt}}
\multiput(182.17,193.00)(-2.000,2.132){2}{\rule{0.400pt}{0.450pt}}
\put(179.67,197){\rule{0.400pt}{0.964pt}}
\multiput(180.17,197.00)(-1.000,2.000){2}{\rule{0.400pt}{0.482pt}}
\put(178.17,201){\rule{0.400pt}{0.900pt}}
\multiput(179.17,201.00)(-2.000,2.132){2}{\rule{0.400pt}{0.450pt}}
\put(176.67,205){\rule{0.400pt}{0.723pt}}
\multiput(177.17,205.00)(-1.000,1.500){2}{\rule{0.400pt}{0.361pt}}
\put(175.67,208){\rule{0.400pt}{0.964pt}}
\multiput(176.17,208.00)(-1.000,2.000){2}{\rule{0.400pt}{0.482pt}}
\put(174.67,216){\rule{0.400pt}{0.723pt}}
\multiput(175.17,216.00)(-1.000,1.500){2}{\rule{0.400pt}{0.361pt}}
\put(176.0,212.0){\rule[-0.200pt]{0.400pt}{0.964pt}}
\put(174.67,231){\rule{0.400pt}{0.723pt}}
\multiput(174.17,231.00)(1.000,1.500){2}{\rule{0.400pt}{0.361pt}}
\put(175.0,219.0){\rule[-0.200pt]{0.400pt}{2.891pt}}
\put(175.67,238){\rule{0.400pt}{0.964pt}}
\multiput(175.17,238.00)(1.000,2.000){2}{\rule{0.400pt}{0.482pt}}
\put(176.67,242){\rule{0.400pt}{0.723pt}}
\multiput(176.17,242.00)(1.000,1.500){2}{\rule{0.400pt}{0.361pt}}
\put(178.17,245){\rule{0.400pt}{0.900pt}}
\multiput(177.17,245.00)(2.000,2.132){2}{\rule{0.400pt}{0.450pt}}
\put(179.67,249){\rule{0.400pt}{0.964pt}}
\multiput(179.17,249.00)(1.000,2.000){2}{\rule{0.400pt}{0.482pt}}
\put(181.17,253){\rule{0.400pt}{0.900pt}}
\multiput(180.17,253.00)(2.000,2.132){2}{\rule{0.400pt}{0.450pt}}
\put(183.17,257){\rule{0.400pt}{0.700pt}}
\multiput(182.17,257.00)(2.000,1.547){2}{\rule{0.400pt}{0.350pt}}
\multiput(185.61,260.00)(0.447,0.685){3}{\rule{0.108pt}{0.633pt}}
\multiput(184.17,260.00)(3.000,2.685){2}{\rule{0.400pt}{0.317pt}}
\put(188.17,264){\rule{0.400pt}{0.900pt}}
\multiput(187.17,264.00)(2.000,2.132){2}{\rule{0.400pt}{0.450pt}}
\multiput(190.00,268.61)(0.462,0.447){3}{\rule{0.500pt}{0.108pt}}
\multiput(190.00,267.17)(1.962,3.000){2}{\rule{0.250pt}{0.400pt}}
\multiput(193.61,271.00)(0.447,0.685){3}{\rule{0.108pt}{0.633pt}}
\multiput(192.17,271.00)(3.000,2.685){2}{\rule{0.400pt}{0.317pt}}
\multiput(196.00,275.60)(0.481,0.468){5}{\rule{0.500pt}{0.113pt}}
\multiput(196.00,274.17)(2.962,4.000){2}{\rule{0.250pt}{0.400pt}}
\multiput(200.00,279.60)(0.481,0.468){5}{\rule{0.500pt}{0.113pt}}
\multiput(200.00,278.17)(2.962,4.000){2}{\rule{0.250pt}{0.400pt}}
\multiput(204.00,283.61)(0.685,0.447){3}{\rule{0.633pt}{0.108pt}}
\multiput(204.00,282.17)(2.685,3.000){2}{\rule{0.317pt}{0.400pt}}
\multiput(208.00,286.60)(0.481,0.468){5}{\rule{0.500pt}{0.113pt}}
\multiput(208.00,285.17)(2.962,4.000){2}{\rule{0.250pt}{0.400pt}}
\multiput(212.00,290.60)(0.481,0.468){5}{\rule{0.500pt}{0.113pt}}
\multiput(212.00,289.17)(2.962,4.000){2}{\rule{0.250pt}{0.400pt}}
\multiput(216.00,294.61)(0.909,0.447){3}{\rule{0.767pt}{0.108pt}}
\multiput(216.00,293.17)(3.409,3.000){2}{\rule{0.383pt}{0.400pt}}
\multiput(221.00,297.60)(0.774,0.468){5}{\rule{0.700pt}{0.113pt}}
\multiput(221.00,296.17)(4.547,4.000){2}{\rule{0.350pt}{0.400pt}}
\multiput(227.00,301.60)(0.627,0.468){5}{\rule{0.600pt}{0.113pt}}
\multiput(227.00,300.17)(3.755,4.000){2}{\rule{0.300pt}{0.400pt}}
\multiput(232.00,305.60)(0.774,0.468){5}{\rule{0.700pt}{0.113pt}}
\multiput(232.00,304.17)(4.547,4.000){2}{\rule{0.350pt}{0.400pt}}
\multiput(238.00,309.61)(1.355,0.447){3}{\rule{1.033pt}{0.108pt}}
\multiput(238.00,308.17)(4.855,3.000){2}{\rule{0.517pt}{0.400pt}}
\multiput(245.00,312.60)(0.920,0.468){5}{\rule{0.800pt}{0.113pt}}
\multiput(245.00,311.17)(5.340,4.000){2}{\rule{0.400pt}{0.400pt}}
\multiput(252.00,316.60)(0.920,0.468){5}{\rule{0.800pt}{0.113pt}}
\multiput(252.00,315.17)(5.340,4.000){2}{\rule{0.400pt}{0.400pt}}
\multiput(259.00,320.61)(1.579,0.447){3}{\rule{1.167pt}{0.108pt}}
\multiput(259.00,319.17)(5.579,3.000){2}{\rule{0.583pt}{0.400pt}}
\multiput(267.00,323.60)(1.066,0.468){5}{\rule{0.900pt}{0.113pt}}
\multiput(267.00,322.17)(6.132,4.000){2}{\rule{0.450pt}{0.400pt}}
\multiput(275.00,327.60)(1.212,0.468){5}{\rule{1.000pt}{0.113pt}}
\multiput(275.00,326.17)(6.924,4.000){2}{\rule{0.500pt}{0.400pt}}
\multiput(284.00,331.61)(1.802,0.447){3}{\rule{1.300pt}{0.108pt}}
\multiput(284.00,330.17)(6.302,3.000){2}{\rule{0.650pt}{0.400pt}}
\multiput(293.00,334.60)(1.358,0.468){5}{\rule{1.100pt}{0.113pt}}
\multiput(293.00,333.17)(7.717,4.000){2}{\rule{0.550pt}{0.400pt}}
\multiput(303.00,338.60)(1.505,0.468){5}{\rule{1.200pt}{0.113pt}}
\multiput(303.00,337.17)(8.509,4.000){2}{\rule{0.600pt}{0.400pt}}
\multiput(314.00,342.60)(1.505,0.468){5}{\rule{1.200pt}{0.113pt}}
\multiput(314.00,341.17)(8.509,4.000){2}{\rule{0.600pt}{0.400pt}}
\multiput(325.00,346.61)(2.695,0.447){3}{\rule{1.833pt}{0.108pt}}
\multiput(325.00,345.17)(9.195,3.000){2}{\rule{0.917pt}{0.400pt}}
\multiput(338.00,349.60)(1.797,0.468){5}{\rule{1.400pt}{0.113pt}}
\multiput(338.00,348.17)(10.094,4.000){2}{\rule{0.700pt}{0.400pt}}
\multiput(351.00,353.60)(1.943,0.468){5}{\rule{1.500pt}{0.113pt}}
\multiput(351.00,352.17)(10.887,4.000){2}{\rule{0.750pt}{0.400pt}}
\multiput(365.00,357.61)(3.365,0.447){3}{\rule{2.233pt}{0.108pt}}
\multiput(365.00,356.17)(11.365,3.000){2}{\rule{1.117pt}{0.400pt}}
\multiput(381.00,360.60)(2.382,0.468){5}{\rule{1.800pt}{0.113pt}}
\multiput(381.00,359.17)(13.264,4.000){2}{\rule{0.900pt}{0.400pt}}
\multiput(398.00,364.60)(2.528,0.468){5}{\rule{1.900pt}{0.113pt}}
\multiput(398.00,363.17)(14.056,4.000){2}{\rule{0.950pt}{0.400pt}}
\multiput(416.00,368.60)(2.967,0.468){5}{\rule{2.200pt}{0.113pt}}
\multiput(416.00,367.17)(16.434,4.000){2}{\rule{1.100pt}{0.400pt}}
\multiput(437.00,372.61)(4.704,0.447){3}{\rule{3.033pt}{0.108pt}}
\multiput(437.00,371.17)(15.704,3.000){2}{\rule{1.517pt}{0.400pt}}
\multiput(459.00,375.60)(3.698,0.468){5}{\rule{2.700pt}{0.113pt}}
\multiput(459.00,374.17)(20.396,4.000){2}{\rule{1.350pt}{0.400pt}}
\multiput(485.00,379.60)(3.990,0.468){5}{\rule{2.900pt}{0.113pt}}
\multiput(485.00,378.17)(21.981,4.000){2}{\rule{1.450pt}{0.400pt}}
\multiput(513.00,383.61)(7.160,0.447){3}{\rule{4.500pt}{0.108pt}}
\multiput(513.00,382.17)(23.660,3.000){2}{\rule{2.250pt}{0.400pt}}
\multiput(546.00,386.60)(5.599,0.468){5}{\rule{4.000pt}{0.113pt}}
\multiput(546.00,385.17)(30.698,4.000){2}{\rule{2.000pt}{0.400pt}}
\multiput(585.00,390.60)(6.915,0.468){5}{\rule{4.900pt}{0.113pt}}
\multiput(585.00,389.17)(37.830,4.000){2}{\rule{2.450pt}{0.400pt}}
\multiput(633.00,394.60)(8.670,0.468){5}{\rule{6.100pt}{0.113pt}}
\multiput(633.00,393.17)(47.339,4.000){2}{\rule{3.050pt}{0.400pt}}
\multiput(693.00,398.61)(17.877,0.447){3}{\rule{10.900pt}{0.108pt}}
\multiput(693.00,397.17)(58.377,3.000){2}{\rule{5.450pt}{0.400pt}}
\multiput(774.00,401.60)(18.613,0.468){5}{\rule{12.900pt}{0.113pt}}
\multiput(774.00,400.17)(101.225,4.000){2}{\rule{6.450pt}{0.400pt}}
\multiput(902.00,405.60)(45.810,0.468){5}{\rule{31.500pt}{0.113pt}}
\multiput(902.00,404.17)(248.620,4.000){2}{\rule{15.750pt}{0.400pt}}
\put(176.0,234.0){\rule[-0.200pt]{0.400pt}{0.964pt}}
\put(65,410){\usebox{\plotpoint}}
\put(65.0,410.0){\rule[-0.200pt]{329.792pt}{0.400pt}}
\put(65,40){\usebox{\plotpoint}}
\put(65.0,40.0){\rule[-0.200pt]{329.792pt}{0.400pt}}
\end{picture}
\caption{\label{figura2} The {\em grim reaper} relative to $e_1$.}
\end{figure}
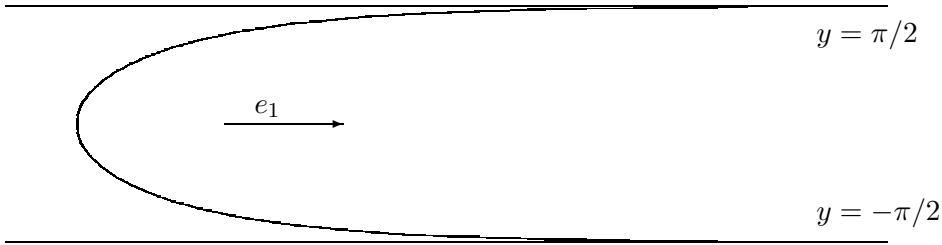


Passing from a single curve to a regular network, the situation
becomes more delicate. 
Every curve of the translating network has to satisfies 
$k^i(x)=\left\langle\boldsymbol{v},\nu^i(x)\right\rangle$.
A result for translating triods can be found 
in~\cite[Lemma~5.8]{mannovtor}:
a closed, unbounded and embedded regular triod 
$\mathbb{R}^2$ self--translating with velocity $\boldsymbol{v}\neq 0$
is composed by
halflines parallel to $\boldsymbol{v}$ or 
translated copies of pieces of the grim reaper relative to $\boldsymbol{v}$, 
meeting at the 3--point with angles of 120 degrees.
Notice that at most one curve is a halfline (Figure~\ref{figura3}).


\begin{figure}[h]
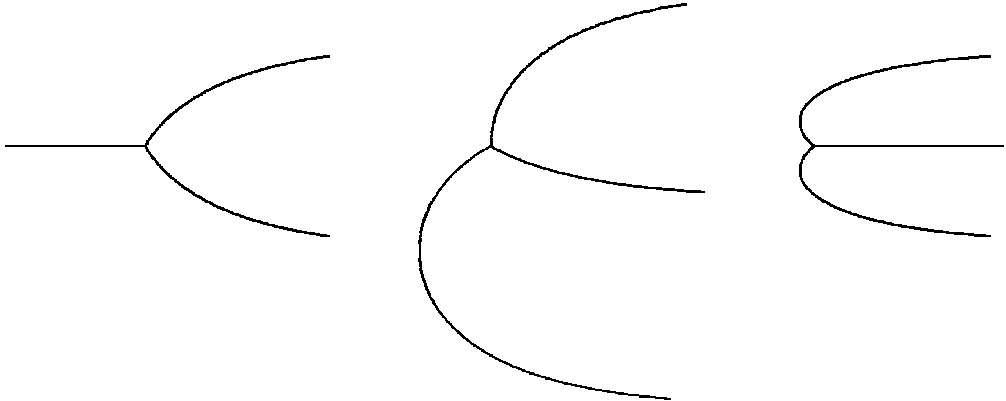

\setlength{\unitlength}{0.240900pt}
\ifx\plotpoint\undefined\newsavebox{\plotpoint}\fi

\caption{\label{figura3} Some examples of translating triods.}
\end{figure}


Among curves there are also 
\textbf{rotating} solutions. 
Suppose indeed that $\gamma$ is of the form 
$\gamma(t,x)=R(t)\eta(x)$ with $R(t)$ a rotation.
The motion equation becomes
$$
k(x)
=\left\langle R'(t)\eta(x),R(t)\nu(x)\right\rangle
=\left\langle R^t(t) R'(t)\eta(x),\nu(x)\right\rangle\,.
$$
We get that $R^t(t) R'(t)$ is constant.
By straightforward computations
one also get that $R(t)$ is a anticlockwise rotation 
by $\omega t$, where $\omega$ is a given constant.
Then 
$$
k(x)=\omega \left\langle\eta(x),\tau(x)\right\rangle\,.
$$
A fascinating example can be found in~\cite{alt1}: the Yin--Yang curve.

\medskip

We have left at last the more significant case: the \textbf{self--similarly shrinking networks}.

Suppose that a solution of the motion by curvature 
evolves homothetically shrinking in time 
with center of homothety the origin, namely 
$\gamma(t,x)=\alpha(t)\eta(x)$ with
$\alpha(t)>0$ and $\alpha'(t)<0$.
Being a solution of the flow 
the curve $\gamma$ satisfies
$k(t,x)=\left\langle\gamma_t(t,x),\nu(t,x)\right\rangle$.
Then
\begin{align*}
k(x)=\alpha(t)\alpha'(t)\left\langle \eta(x),\nu(x)\right\rangle\,.
\end{align*}
We have that $\alpha(t)\alpha'(t)$ is equal to some constant. Up to
rescaling we can suppose $\alpha(t)\alpha'(t)=-1$.
Then for every $t\in (-\infty,0]$ we have  $\alpha(t)=2\sqrt{t-T}$
and  $k(x)=-\left\langle \eta(x),\nu(x)\right\rangle$,
or equivalently 
$\boldsymbol{k}(x)+\eta^\perp(x)=0$.

\begin{defn}\label{shrinkers} A regular $C^2$ open network 
$\mathcal{S}$ union of $n$ curves parametrized by $\eta^i$
is called a \textbf{regular shrinker} if for every curve there holds
\begin{equation}\label{shrinkeq}
\boldsymbol{k}^i + (\eta^i)^\perp=0\,. 
\end{equation}
\end{defn}

\begin{rem}\label{abla}
Every curve of a  regular shrinker satisfies the equation 
$\boldsymbol{k}+ \eta^\perp=0$.
As a consequence
 it must be a piece of a line though the origin or of the
so called \textbf{Abresch--Langer curves}. 
Their classification results
in~\cite{ablang1} imply that any of these non straight pieces is
compact. 
Hence any unbounded curve of a shrinker must be a line or an
halfline pointing towards the origin. Moreover, it also follows that if a
curve contains the origin, then it is a straight line through the
origin or a halfline from the origin.
\end{rem}

By the work of Abresch and Langer~\cite{ablang1} 
it follows that the only regular shrinkers without triple junctions 
(curves) are the lines for the origin and the unit circle. 
There are two shrinkers with one triple junction~\cite{haettenschweiler}: 
the standard triod and the 
\textbf{Brakke spoon}.
The Brakke spoon is a regular shrinker composed by a halfline which intersects a closed curve, 
forming angles of $120$ degrees.
It was first mentioned in~\cite{brakke} as an example of evolving network with a loop shrinking 
down to a point, leaving a halfline 
that then,  in the framework of Brakke flows, vanishes instantaneously. 
Up to rotation, this particular spoon--shaped network is unique~\cite{chenguo}
(Figure~\ref{figuraBrakkespoon}).

\begin{figure}[H]
\begin{center}
\begin{tikzpicture}[scale=0.5]
\draw[color=black,scale=1,domain=-3.141: 3.141,
smooth,variable=\t,shift={(7.5,0)},rotate=0]plot({2.*sin(\t r)},
{2.*cos(\t r)});
\fill(7.5,0) circle (2pt);
\path[font=\small]
(7.65,.2) node[left]{$O$};
\end{tikzpicture}\qquad\qquad
\begin{tikzpicture}[scale=0.3]
\draw[color=black]
(-6.3,0)to[out= 0,in=180, looseness=1](-1.535,0)
(-1.535,0)to[out= 60,in=180, looseness=1] (3.7,3)
(3.7,3)to[out= 0,in=90, looseness=1] (6.93,0)
(-1.535,0)to[out= -60,in=180, looseness=1] (3.7,-3)
(3.7,-3)to[out= 0,in=-90, looseness=1] (6.93,0);
\draw[color=black,dashed](-8,0)to[out= 0,in=180, looseness=1](-6.5,0);
\fill(1,0) circle (3pt);
\path[font=\large]
 (.3,-.35) node[above]{$O$};
\end{tikzpicture}
\end{center}
\begin{caption}{A circle and a Brakke spoon. Together with
a straight line and a standard triod, they are all possible
regular shrinker with at most one triple junction. }\label{figuraBrakkespoon}
\end{caption}
\end{figure}

Also the classification of shrinkers with two triple junctions is complete.
It is not difficult to show~\cite{balhausman2, balhausman3}
that there are only two possible topological shapes for a complete
embedded, regular shrinker: one is the ``lens/fish'' shape and the
other is the shape of the Greek ``Theta'' letter (or ``double cell'').
It is well known that there exist unique (up to a rotation)
lens--shaped or fish--shaped, embedded, regular shrinkers which are
symmetric with respect to a line through the origin of $\R^2$~\cite{chenguo,schnurerlens} (Figure~\ref{fishfig}). 
Instead, there are no regular $\Theta$--shaped shrinkers~\cite{balhausman}.

\begin{figure}[h]
\begin{center}
\begin{tikzpicture}[scale=0.3]
\draw[color=black]
(-3.035,0)to[out= 60,in=180, looseness=1] (2.2,2.8)
(2.2,2.8)to[out= 0,in=120, looseness=1] (7.435,0)
(2.2,-2.7)to[out= 0,in=-120, looseness=1] (7.435,0)
(-3.035,0)to[out= -60,in=180, looseness=1] (2.2,-2.7);
\draw[color=black]
(-7,0)to[out= 0,in=180, looseness=1](-3.035,0)
(7.435,0)to[out= 0,in=180, looseness=1](11.4,0);
\draw[color=black,dashed]
(-9,0)to[out= 0,in=180, looseness=1](-7,0)
(11.4,0)to[out= 0,in=180, looseness=1](13.4,0);
\fill(2.2,0) circle (3pt);
\path[font=\small]
(1.5,-.35) node[above]{$O$};
\draw[color=black,scale=4,shift={(7,0)}]
(-0.47,0)to[out= 20,in=180, looseness=1](1.5,0.65)
(1.5,0.65)to[out= 0,in=90, looseness=1] (2.37,0)
(-0.47,0)to[out= -20,in=180, looseness=1](1.5,-0.65)
(1.5,-0.65)to[out= 0,in=-90, looseness=1] (2.37,0);
\draw[white, very thick,scale=4,shift={(7,0)}]
(-0.47,0)--(-.150,-0.13)
(-0.47,0)--(-.150,0.13);
\draw[color=black,scale=4,shift={(7,0)}]
(-.150,0.13)to[out= -101,in=90, looseness=1](-.18,0)
(-.18,0)to[out= -90,in=101, looseness=1](-.150,-0.13);
\draw[black,scale=4,shift={(7,0)}]
(-.150,0.13)--(-1.13,0.98);
\draw[black,scale=4,shift={(7,0)}]
(-.150,-0.13)--(-1.13,-0.98);
\draw[black, dashed,scale=4,shift={(7,0)}]
(-1.13,0.98)--(-1.50,1.31);
\draw[black, dashed,scale=4,shift={(7,0)}]
(-1.13,-0.98)--(-1.50,-1.31);
\fill(28.05,0) circle (3pt);
\path[font=\small]
(29,-.6) node[above]{$O$};
\end{tikzpicture}
\end{center}
\begin{caption}{The \textbf{standard lens} is a shrinker with two triple junctions 
symmetric with respect to two perpendicular axes,
composed by two halflines pointing the origin, posed on a symmetry axis and opposite with respect to the other.
Each halfline intersects two equal curves forming  an angle of $120$ degrees.
The \textbf{fish} is a shrinker with the same topology of the standard lens, but 
symmetric with respect to only one axis.
The two halfines, pointing the origin, intersect two different curves, forming angles of $120$ degrees.\label{fishfig}}
\end{caption}
\end{figure}

\begin{figure}[h]
\begin{center}
\begin{tikzpicture}[scale=0.23, rotate=60]
\fill[color=black](0,0) circle (4pt);   
\draw[color=black, scale=1.32]
(-1.37,2.38)to[out= 120,in=-60, looseness=1](-2.5,4.33)
(-1.37,-2.38)to[out= -120,in=60, looseness=1](-2.5,-4.33);
\draw[color=black]
(2.75,0)to[out= 0,in=180, looseness=1](5,0)
(-1.37,2.38)to[out= 120,in=-60, looseness=1](-2.5,4.33)
(-1.37,-2.38)to[out= -120,in=60, looseness=1](-2.5,-4.33)
(-1.37,-2.38)to[out= 120,in=-120, looseness=1](-1.37,2.38)
(2.75,0)to[out= 120,in=0, looseness=1](-1.37,2.38)
(2.75,0)to[out= -120,in=0, looseness=1](-1.37,-2.38);
\draw[color=black,dashed]
(5,0)to[out= 0,in=180, looseness=1](6,0);
\draw[color=black,dashed, scale=1.32]
(-2.5,4.33)to[out= 120,in=-60, looseness=1](-3,5.19)
(-2.5,-4.33)to[out= -120,in=60, looseness=1](-3,-5.19);
\end{tikzpicture}\quad\qquad
\begin{tikzpicture}[scale=0.4, rotate=90]
 \fill[color=black](0,0) circle (2pt);   
 \draw[color=black,shift={(0,0)}]
(2.8,0)to[out= 0,in=180, looseness=1](4,0);
 \draw[color=black,dashed,shift={(0,0)}]
(4,0)to[out= 0,in=180, looseness=1](5,0);
(-1.37,2.38)to[out= 120,in=-60, looseness=1](-2.5,4.33)
(-1.37,-2.38)to[out= -120,in=60, looseness=1](-2.5,-4.33)
(-1.37,-2.38)to[out= 120,in=-120, looseness=1](-1.37,2.38)
(2.8,0)to[out= 120,in=0, looseness=1](-1.37,2.38)
(2.8,0)to[out= -120,in=0, looseness=1](-1.37,-2.38);
\draw[color=black,shift={(-0.05,0)},scale=4]
(-.150,0.13)to[out= -101,in=90, looseness=1](-.18,0)
(-.18,0)to[out= -90,in=101, looseness=1](-.150,-0.13);
\draw[black,shift={(-0.05,0)},scale=4]
(-.150,0.13)--(-0.63,0.45);
\draw[black,shift={(-0.05,0)},scale=4]
(-.150,-0.13)--(-0.63,-0.45);
\draw[color=black,shift={(-0.05,0)},scale=4]
(-.150,0.13)to[out= 20,in=180, looseness=1](0.325,0.2)
(0.325,0.2)to[out= 0,in=120, looseness=1] (0.725,0)
(-.150,-0.13)to[out= -20,in=180, looseness=1](0.325,-0.2)
(0.325,-0.2)to[out= 0,in=-120, looseness=1] (0.725,0);
\draw[black,dashed,shift={(-0.05,0)},scale=4]
(-0.63,0.45)--(-0.7,0.5);
\draw[black,dashed,shift={(-0.05,0)},scale=4]
(-0.63,-0.45)--(-0.7,-0.5);
 \fill[color=white](0,-3.225) circle (2pt);  
\end{tikzpicture}\quad\quad
\begin{tikzpicture}[scale=0.5, rotate=90]
\draw[color=black]
(-2.5,0)to[out= 0,in=180, looseness=1](-1.15,0)
(1.15,0)to[out= 0,in=180, looseness=1](2.5,0)
(0,2.5)to[out= -90,in=90, looseness=1](0,1.15)
(0,-2.5)to[out= 90,in=-90, looseness=1](0,-1.15)
(-1.15,0)to[out= 60,in=-150, looseness=1](0,1.15)
(-1.15,0)to[out= -60,in=150, looseness=1](0,-1.15)
(1.15,0)to[out= 1200,in=-30, looseness=1](0,1.15)
(1.15,0)to[out= -120,in=30, looseness=1](0,-1.15);
\draw[color=black,dashed]
(-3,0)to[out= 0,in=180, looseness=1](-2.5,0)
(2.5,0)to[out= 0,in=180, looseness=1](3,0)
(0,3)to[out= -90,in=90, looseness=1](0,2.5)
(0,-3)to[out= 90,in=-90, looseness=1](0,-2.5);
\fill(0,0) circle (1.5pt);
\end{tikzpicture}\quad\quad
\begin{tikzpicture}[scale=0.55, rotate=90]
\draw[color=black]
(0.76,0)to[out= 0,in=180, looseness=1](2.5,0)
(0.23,-0.72)to[out=-72,in=108, looseness=1](0.77,-2.37)
(0.23,0.72)to[out=72,in=-108, looseness=1](0.77,2.37)
(-0.61,-0.44)to[out=-144,in=36, looseness=1](-2,-1.46)
(-0.61,0.44)to[out=144,in=-36, looseness=1](-2,1.46)
(0.76,0)to[out= 120,in=-48, looseness=1](0.23,0.72)
to[out= -162,in=24, looseness=1](-0.61,0.44)
to[out= -96,in=96, looseness=1](-0.61,-0.44)
(0.76,0)to[out= -120,in=48, looseness=1](0.23,-0.72)
to[out= 162,in=-24, looseness=1](-0.61,-0.44);
\draw[color=black,dashed]
(2.5,0)to[out= 0,in=180, looseness=1](3,0)
(0.77,-2.37)to[out=-72,in=108, looseness=1](0.92,-2.85)
(0.77,2.37)to[out=72,in=-108, looseness=1](0.92,2.85)
(-2,-1.46)to[out=-144,in=36, looseness=1](-2.42,-1.76)
(-2,1.46)to[out=144,in=-36, looseness=1](-2.42,1.76);
\fill(0,0) circle (1.5pt);
 \fill[color=white](-2.6875,0) circle (2pt);  
\end{tikzpicture}
\end{center}
\begin{caption}{The regular shrinkers with a single bounded region.\label{shrinoti}}
\end{caption}
\end{figure}

The classification of  (embedded) regular shrinkers is completed for
the shrinkers with a single bounded region~\cite{chen_guo,chenguo,schnurerlens,balhausman}, see Figure~\ref{shrinoti}.

Several questions (also of independent interest) 
arise in trying to classify the regular shrinkers. 
We just mention an open question: does there exist a regular shrinker
with more than five unbounded halflines? 

Numerical computations, partial results and conjectures can be found in~\cite{haettenschweiler}.

\section{Integral estimates}\label{estimates}

A good way to understand what happens during the evolution
of a network by curvature is to describe the changing in time of 
the geometric quantities related to the network.
For instance we can write the evolution law of the
length of the curves or of area enclosed by the curves.
In several situations estimating the evolution of the curvature 
has revealed a winning strategy to pass from short time to long time 
existence results.

Differently from the case of the curve shortening flow (and of the mean curvature flow)
here to obtain our a priori estimates we cannot use the maximum principle
and a comparison principle is not valid because of the presence of junctions.
Therefore integral estimates are computed in~\cite[Section~3]{mannovtor}
in~\cite[Section~5]{mannovplusch}
in the case of a triod and a regular network, respectively.
An outline for the estimates appeared in~\cite[Section 7]{Ilnevsch}, where the authors
consider directly the evolution $\gamma_t=k\nu+\lambda\tau$. 
We summarise here these calculations focusing on the easier cases.

\medskip 

Form now on we suppose that 
all the derivatives of the functions  that appear exist.

\medskip 

We start showing that if a curve moves by curvature, then its time derivative $\partial_t$ and the arclength derivative
$\partial_s$ do not commute.

We have already mentioned that the motion by curvature $\gamma^\perp_t=\boldsymbol{k}$
can be written as 
$$
\gamma_t=k\nu+\lambda\tau\,,
$$
for some continuous function $\lambda$.

\begin{lem} If $\gamma$ is a curve moving by $\gamma_t=k\nu+\lambda\tau$,
then we have the following commutation rule:
\begin{equation}\label{commut}
\dert\ders=\ders\dert +
(k^2 -\lambda_s)\ders\,.
\end{equation}
\end{lem}
\begin{proof}
Let $f:[0,1]\times[0,T)\to\R$ be a smooth function, then
\begin{align*}
\dert\ders f - \ders\dert f =&\, \frac{f_{tx}}{\vert\gamma_x\vert} -
\frac{\langle \gamma_x\,\vert\,\gamma_{xt}\rangle f_x}
{\vert\gamma_x\vert^3}
- \frac{f_{tx}}{\vert\gamma_x\vert} = - {\langle
  \tau\,\vert\,\partial_s\gamma_t\rangle}\partial_sf\\
=&\, - {\langle\tau\,\vert\,\partial_s(\lambda\tau+k\nu)\rangle}\partial_sf=
 (k^2 - \lambda_s)\ders f
\end{align*}
and the formula is proved.
\end{proof}

In all this section we will  consider a
$C^\infty$ solution of the special flow.
Hence each curve is moving by  
$$
\gamma^i_t(t,x)=\frac{\gamma_{xx}^{i}\left(t,x\right)}{\left|\gamma_{x}^{i}\left(t,x\right)\right|^{2}}\,,
$$
and 
$\lambda=\frac{\langle\gamma_{xx}\,\vert\,\gamma_x\rangle}{{\vert\gamma_x\vert}^3}$.

Using the rule in the previous lemma we can compute
\begin{align}
\dert\tau=&\,
\dert\ders\gamma=\ders\dert\gamma+(k^2-\lambda_s)\ders\gamma =
\ders(\lambda\tau+k\nu)+(k^2-\lambda_s)\tau =
(k_s+k\lambda)\nu\label{derttau}\,,\\
\dert\nu=&\, \dert({\mathrm R}\tau)={\mathrm
R}\,\dert\tau=-(k_s+k\lambda)\tau\label{dertdinu}\,,\\
\dert k=&\, \dert\langle \ders\tau\,|\, \nu\rangle=
\langle\dert\ders\tau\,|\, \nu\rangle\label{dertdik}
= \langle\ders\dert\tau\,|\, \nu\rangle +
(k^2-\lambda_s)\langle\ders\tau\,|\, \nu\rangle\\
=&\, \ders\langle\dert\tau\,|\, \nu\rangle + k^3-k\lambda_s =
\ders(k_s+k\lambda) + k^3-k\lambda_s\nonumber\\
=&\, k_{ss}+k_s\lambda + k^3\nonumber\,,\\
\dert\lambda =&\, -\dert\partial_x\frac{1}{\vert\gamma_x\vert}=
\partial_x \frac{\langle\gamma_x\,\vert\,\gamma_{tx}\rangle}
{\vert\gamma_x\vert^3}=
\partial_x \frac{\langle\tau\,\vert\,\ders (\lambda\tau+k\nu)\rangle}
{\vert\gamma_x\vert}=\partial_x \frac{(\lambda_s - k^2)}
{\vert\gamma_x\vert}\label{dertdilamb}\\
=&\, \ders(\lambda_s - k^2) -\lambda(\lambda_s -
k^2)=\lambda_{ss} -\lambda\lambda_s - 2kk_s +\lambda k^2\,.\nonumber
\end{align}

\subsection{Evolution of length and volume}\label{evolunghezza}

We now compute the evolution in time of the total length.

By the commutation formula~\eqref{commut} the
time derivative of the measure $ds$ on any curve $\gamma^i$ of the network is
given by the measure $(\lambda_s^i-(k^i)^2)\,ds$.
Then the evolution law for the length of one curve is
\begin{equation}\label{evoluzionelunghezza}
\frac{dL^i(t)}{dt}= \frac{d\,}{dt}\int_{\gamma^i(\cdot,t)}1\, ds
=\int_{\gamma^i(\cdot,t)}(\lambda^i_s-(k^i)^2)\,ds
=\lambda^i(1,t)-\lambda^i(0,t)-\int_{\gamma^i(\cdot,t)}(k^i)^2\,ds\,.
\end{equation}

We remind that by relation~\eqref{eq:cond2}
the contributions of $\lambda^{pi}$ at every $3$--point $O^p$ vanish.
Suppose that the network has $l$ end--points on the boundary of $\Omega$.
With a little abuse of notation we call
$\lambda(t,P^r)$ the tangential velocity at the end--point $P^r$ for any $r\in\{1,2,\dots,l\}$.
Since the total length is  the sum of the lengths of all the curves,
we get 
$$
\frac{dL(t)}{dt}=\sum_{r=1}^l\lambda(t,P^r)-\int_{\mathcal{N}_t}k^2\,ds\,,
$$
In particular if the end--points $P^r$ of the network are fixed during the evolution
all the terms $\lambda(t,P^r)$ are zero and we have 
\begin{equation}\label{evolength}
\frac{dL(t)}{dt}=-\int_{\mathcal{N}_t}k^2\,ds\,.
\end{equation}
The total length $L(t)$ is decreasing in time and uniformly bounded above 
by the length of the initial network.


We now discuss the behavior of the area of the regions enclosed by some curves of 
the evolving regular network. Let us suppose that a region
${\mathcal{A}}(t)$ is bounded by $m$ curves
 $\gamma^1,\gamma^2,\dots,\gamma^m$ and let $A(t)$ be its area.
We call \textbf{loop} $\ell$ the union of these $m$ curves. 
The loop $\ell$ can be regarded as a single piecewise $C^2$ closed curve
parametrized anticlockwise
(possibly after reparametrization of the curves
that composed it). 
Hence the curvature of $\ell$ is
positive at the convexity points of the boundary of ${\mathcal{A}}(t)$. 
Then we have
\begin{equation}\label{areauno}
A'(t)=-\sum_{i=1}^m\int_{\gamma^i}\langle x_t\,\vert\,\nu\rangle\,ds=-\sum_{i=1}^m\int_{\gamma^i}\langle k\nu\,\vert\,\nu\rangle\,ds
=-\sum_{i=1}^m\int_{\gamma^i} k\,ds
=-\sum_{i=1}^m\Delta\theta_i\,,
\end{equation}
where $\Delta\theta_i$ is the difference in the angle between the unit
tangent vector $\tau$ and the unit coordinate vector $e_1\in\R^2$ at the
final and initial point of the curve $\gamma^i$. 
Indeed supposing the
unit tangent vector of the curve $\gamma^i$ ``lives'' in the second quadrant
of $\R^2$ (the other cases are analogous) there holds 
$$
\partial_s\theta_i=\partial_s\arccos \langle\tau\,\vert\,e_1\rangle
=-\frac{\langle\tau_s\,\vert\,e_1\rangle}{\sqrt{1-\langle\tau\,\vert\,e_1\rangle^2}}=k\,,
$$
so
$$
A'(t)=-\sum_{i=1}^m\int_{\gamma^i}\partial_s\theta_i\,ds
=-\sum_{i=1}^m\Delta\theta_i\,.
$$
Considering that
the curves $\gamma^i$ form angles
of $120$ degrees, we have
$$
m\pi/3+\sum_{i=1}^m\Delta\theta_i=2\pi\,.
$$
We then obtain the equality (see~\cite{vn})
\begin{equation}\label{areaevolreg}
A'(t)=-(2-m/3)\pi\,.
\end{equation}

An immediate consequence of \eqref{areaevolreg} is that the area of every region
 bounded by the curves of the network evolves linearly. More precisely it
increases if the region has more than six edges, it is constant with six
edges and it decreases if its edges are less than six. 
This implies that if less than six curves of the initial network
enclose a region of area $A_0$,
 then the maximal time $T$ of existence of a smooth flow is finite and
\begin{equation}\label{stimaT0}
T\leq \frac{A_0}{(2-m/3)\pi}\leq\frac{3A_0}{\pi}\,. 
\end{equation}

\subsection{Evolution of the curvature and its derivatives}

We want to estimate the $L^2$ norm of the curvature and its derivatives, that 
will result crucial in the analysis of the motion.
The main consequence of these computation indeed is that the flow of a regular smooth network with ``controlled'' end--points exists smooth as long as
the curvature stays bounded and none of the lengths of the curves goes to zero (Theorem~\ref{curvexplod}). 

\medskip

We consider a regular $C^\infty$ network $\mathcal{N}_t$ in $\Omega$, 
composed by $n$ curves $\gamma^{i}$ with $m$ triple--points
$O^1, O^2,\dots, O^m$ and $l$ end--points $P^1, P^2,\dots, P^l$. 
We suppose that it is a 
$C^\infty$  solution of the system~\eqref{problema-nogauge-general}.
We assume that either the end--points are fixed (the Dirichlet boundary
condition in~\eqref{problema-nogauge-general} is satisfied) or that there 
exist uniform (in time) constants $C_j$, for every $j\in\NN$, such that 
\begin{equation}\label{endsmooth}
\vert\partial_s^jk(P^r,t)\vert+\vert\partial_s^j\lambda(t,P^r)\vert\leq C_j\,,
\end{equation}
for every $t\in[0,T)$ and $r\in{1,2,\dots,l}$.
This second possibility will allow us to localise the estimates if  needed.


We are now ready to compute
$\frac{d\,}{dt} \int_{{\mathcal{N}_t}} |k|^2\,ds$.
We get

\begin{equation}
\frac{d\,}{dt} \int_{{\mathcal{N}_t}} |k|^2\,ds
= \, 2\int_{{\mathcal{N}_t}} k\, \dert k\,ds +
\int_{{\mathcal{N}_t}} | k|^2(\lambda_s -k^2)\,ds\,.\nonumber\\ 
\end{equation}

Using that $\partial_t k=k_{ss}+k_s\lambda+k^3$ we get 
\begin{equation}
\frac{d\,}{dt} \int_{{\mathcal{N}_t}} |k|^2\,ds
= \, \int_{{\mathcal{N}_t}} 2k\, k_{ss} +2\lambda k\, k_{s} + k^2 \lambda_s  +k^4\,ds
=\int_{{\mathcal{N}_t}} 2k\, k_{ss} +  \ders(\lambda\,k^2)+k^4\,ds\,.\nonumber\\  
\end{equation}

Integrating by parts and 
estimating the contributions given by the end--points $P^r$ by means of
assumption~\eqref{endsmooth} we can write

\begin{align}
\frac{d\,}{dt} \int_{{\mathcal{N}_t}} |k|^2\,ds
=&\, -2\int_{{\mathcal{N}_t}} \vert k_{s}\vert^2\,ds 
+ \int_{{\mathcal{N}_t}} \ders(\lambda\,k^2)\,ds 
+\int_{{\mathcal{N}_t}}k^4\,ds\nonumber\\ &\,
- 2\sum_{p=1}^m\sum_{i=1}^3  k^{pi}\,k_{s}^{pi}\, \biggr\vert_{\text{{ at the $3$--point $O^p$}}} 
+ 2\sum_{r=1}^l k^r\,k^r_{s}\, \biggr\vert_{\text{{ at the end--point} $P^r$}}\nonumber\\ 
\leq&\, -2\int_{{\mathcal{N}_t}} \vert k_{s}\vert^2\,ds  
+ \int_{\mathcal{N}_t} k^4\,ds + lC_0C_{1}\nonumber\\ 
&\,- \sum_{p=1}^m\sum_{i=1}^3 2 k^{pi}\,k_{s}^{pi}+\lambda^{pi}\vert k^{pi}\vert^2\, 
\biggr\vert_{\text{{ at  the $3$--point $O^p$}}}\,.
\end{align} 

Then recalling relation~\eqref{eq:orto} at the $3$--points we have 
$$
\sum_{i=1}^3 k^{i}k^{i}_s+\lambda^{i}|k^{i}|^2\,=0\,.
$$
Substituting it above we lower the maximum order of the space derivatives of the curvature in the $3$--point terms
\begin{equation}\label{ksoltanto} 
\frac{d\,}{dt} \int_{{\mathcal{N}_t}} k^2\,ds 
\leq -2\int_{{\mathcal{N}_t}} \vert k_s\vert^2\,ds  
+ \int_{\mathcal{N}_t} k^4\,ds  
+\sum_{p=1}^m\sum_{i=1}^3 \lambda^{pi}|k^{pi}|^2\,\biggr\vert_{\text{{
      at      the $3$--point $O^p$}}} + lC_0C_1\,.
\end{equation} 

We notice that we can estimate the boundary terms at each $3$--point of the form
$\sum_{i=1}^3 \lambda^{i}|k^{i}|^2$ by 
$\sum_{i=1}^3 \lambda^{i}|k^{i}|^2\leq\Vert k^3\Vert_{L^\infty}$
(see~\cite[Remark~3.9]{mannovtor}).
Hence 
\begin{equation}\label{ksoltanto1} 
\frac{d\,}{dt} \int_{{\mathcal{N}_t}} k^2\,ds 
\leq -2\int_{{\mathcal{N}_t}} \vert k_s\vert^2\,ds  
+ \int_{\mathcal{N}_t} k^4\,ds  
+\Vert k\Vert^3_{L^\infty} + lC_0C_1\,.
\end{equation} 

From now on we do not use any geometric property of our problem. 
We suppose that the lengths of curves 
of the networks are equibounded from below by some positive value.
We reduce to 
estimate the $L^4$ and $L^\infty$ norm of the curvature of any curve $\gamma^i$, 
seen as a Sobolev function defined on the interval $[0,L(\gamma^i)]$.

\begin{lem}\label{finestima}
Let $0<\mathrm L<+\infty$ and $u\in C^\infty([0,\mathrm L],\mathbb{R})$.
Then there exists a uniform constant $C$, depending on $\mathrm L$, such that 
$$
\Vert u\Vert_{L^4}^4+\Vert u\Vert^3_{L^\infty}-2 \Vert u'\Vert_{L^2}^2 
\leq C\left(\Vert u\Vert_{L^2}^2 +1 \right)^3\,.
$$
\end{lem}
\begin{proof}
The key estimates of the proof are 
Gagliardo--Nirenberg interpolation 
inequalities~\cite[Section~3, pp.~257--263]{nirenberg1} written
in the form (see also~\cite[Proposition~3.11]{mannovtor})

\begin{align*}
{\Vert u\Vert}_{L^p}    & \leq C_{p}      
{\Vert u'\Vert}_{L^2}^{\frac12-\frac1p}       
{\Vert u\Vert}_{L^2}^{\frac12+\frac1p}      +      
\frac{B_{p}}{{\mathrm L}^{\frac12-\frac1p}}{\Vert u\Vert}_{L^2}   \,,\\
{\Vert u\Vert}_{L^\infty}  &   \leq C
{\Vert u'\Vert}_{L^2}^{\frac 12}       
{\Vert u\Vert}_{L^2}^{\frac 12}+      
 \frac{B}{{\mathrm L}^{\frac 12}}{\Vert         
u\Vert}_{L^2}\,.
\end{align*}  

We first focus on the term $\Vert u\Vert_{L^4}^4$.
We have
\begin{equation*}
\Vert u\Vert_{L^4}\leq C\left(\Vert u'\Vert^{1/4}_{L^2}\Vert u\Vert^{3/4}_{L^2}
+\frac{\Vert u\Vert_{L^2}}{{\mathrm L}^{\frac 12}}\right)\,,
\end{equation*}
and so
\begin{equation*}
\Vert u\Vert^4_{L^4}\leq \widetilde{C}\left(
\Vert u'\Vert_{L^2}\Vert u\Vert^{3}_{L^2}+\Vert u\Vert^4_{L^2}
\right)\,.
\end{equation*}
Using Young inequality
\begin{equation}\label{primopezzo}
\Vert u\Vert^4_{L^4}\leq \widetilde{C}\left(
\varepsilon\Vert u'\Vert^2_{L^2}+c_\varepsilon\Vert u\Vert^{6}_{L^2}+\Vert u\Vert^4_{L^2}
\right)\,.
\end{equation}
Similarly we estimate the term $\Vert u\Vert^3_{L^\infty}$
by
\begin{equation}\label{secondopezzo}
\Vert u\Vert^3_{L^\infty}\leq C
\left(
\Vert u'\Vert^{\frac32}_{L^2}\Vert u\Vert^{\frac32}_{L^2}+\Vert u\Vert^3_{L^2}
\right)
\leq C
\left(
\varepsilon\Vert u'\Vert^{2}_{L^2}+c_\varepsilon\Vert u\Vert^{6}_{L^2}+\Vert u\Vert^3_{L^2}
\right)\,.
\end{equation}
Putting~\eqref{primopezzo} and~\eqref{secondopezzo} together 
and choosing appropriately $\varepsilon$ we obtain
$$
\Vert u\Vert_{L^4}^4+\Vert u\Vert^3_{L^\infty}-2 \Vert u'\Vert_{L^2}^2 
\leq \tilde{c}_\varepsilon \Vert u\Vert^{6}_{L^2}+\widetilde{C} \Vert u\Vert^4_{L^2}
+C \Vert u\Vert^3_{L^2}
\leq C\left(\Vert u\Vert_{L^2}^2 +1 \right)^3\,.
$$
\end{proof}

Applying Lemma~\ref{finestima} to the curvature $k^i$ of each curve $\gamma^i$
of the network the estimate~\eqref{ksoltanto1} becomes 
\begin{equation}\label{evolint999} 
\left\vert\frac{d\,}{dt} \int_{{\mathcal{N}_t}} k^2\,ds\right\vert\leq  C\left(\int_{{\mathcal{N}_t}} k^2\,ds\right)^{3} + C +lC_0C_1\,. 
\end{equation}

\medskip 


The aim now is to repeat the 
previous computation for $\ders^j k$ with $j\in\NN$.

Although the calculations are much harder, it is possible to conclude that 
for every even $j\in\NN$ there holds
$$
\int_{\mathcal{N}_t} |\ders^j k|^2\,ds \leq C\int_{0}^t\left(\int_{\mathcal{N}_\xi}   k^2\,ds\right)^{2j+3}\,d\xi 
+ C\left(\int_{\mathcal{N}_t}   k^2\,ds\right)^{2j+1} + Ct + lC_jC_{j+1}t+ C\,. 
$$ 

Passing from integral to $L^\infty$ estimates
we have the following proposition.

\begin{prop}\label{pluto1000} 
If assumption~\eqref{endsmooth} holds, the lengths of all the curves
are uniformly positively bounded from below and the $L^2$ norm of $k$
is uniformly bounded on $[0,T)$, then the curvature of $\mathcal{N}_t$  and
all its space derivatives are uniformly bounded in the same time
interval  by some constants depending only on the $L^2$ integrals 
of the space derivatives of $k$ on the initial network $\mathcal{N}_0$. 
\end{prop}

We now derive  a second set of estimates  
where everything is controlled -- still under the
assumption~\eqref{endsmooth} -- only by the $L^2$ norm of the
curvature and the inverses of the lengths of the curves at time
zero.
 
As before we consider the $C^\infty$ special curvature flow $\mathcal{N}_t$ of a smooth network
$\mathcal{N}_0$ in the time interval $[0,T)$, composed by $n$ curves $\gamma^{i}(\cdot,t):[0,1]\to\overline{\Omega}$ with $m$ triple junctions $O^1, O^2,\dots, O^m$ and $l$ end--points $P^1, P^2,\dots, P^l$, satisfying assumption~\eqref{endsmooth}.

As shown above,
the evolution equations for the lengths of the $n$ curves are given by 
$$
\frac{dL^i(t)}{dt}=\lambda^i(1,t)-\lambda^i(0,t)-\int_{\gamma^i(\cdot,t)}k^2\,ds\,.
$$
Then, proceeding as in the 
computations above, we get
\begin{align*} 
\frac{d\,}{dt} \left(\int_{\mathcal{N}_t} k^2\,ds 
+ \sum_{i=1}^n\frac{1}{L^i}\right)\leq &\, 
-2\int_{{\mathcal{N}_t}} k_{s}^2\,ds 
+ \int_{{\mathcal{N}_t}} k^4\,ds 
+6m\Vert k\Vert_{L^\infty}^3 +lC_0C_1
- \sum_{i=1}^n\frac{1}{(L^i)^2}\frac{dL^i}{dt}\\
=&\, -2\int_{{\mathcal{N}_t}} k_{s}^2\,ds 
+ \int_{{\mathcal{N}_t}} k^4\,ds 
+6m\Vert k\Vert^3_{L^\infty}+lC_0C_1\\
&\,-\sum_{i=1}^n\frac{\lambda^i(1,0)-\lambda^i(0,t)
+\int_{\gamma^i(\cdot,t)}k^2\,ds}{(L^i)^2}\\
\leq&\,-2\int_{{\mathcal{N}_t}} k_{s}^2\,ds 
+ \int_{{\mathcal{N}_t}} k^4\,ds +6m\Vert k\Vert^3_{L^\infty}+lC_0C_1\\
&\,+2\sum_{i=1}^n\frac{\Vert k\Vert_{L^\infty}+C_0}{(L^i)^2} +\sum_{i=1}^n\frac{\int_{\mathcal{N}_t}k^2\,ds}{(L^i)^2}\\
\leq&\,-2\int_{{\mathcal{N}_t}} k_{s}^2\,ds 
+ \int_{{\mathcal{N}_t}} k^4\,ds 
+ (6m+2n/3)\Vert k\Vert_{L^\infty}^3+lC_0C_1+2nC_0^3/3\\ 
&\,+ \frac{n}{3}\left(\int_{\mathcal{N}_t}k^2\,ds\right)^3 
+ \frac{2}{3}\sum_{i=1}^n\frac{1}{(L^i)^3}\,,
\end{align*} 
where we used Young inequality in the last passage.
Proceeding as before, but
 keeping track of the terms where the inverse of the length appear,
it is possible to obtain
\begin{align}\label{contistimaL}
\frac{d\,}{dt} \left(\int_{\mathcal{N}_t} k^2\,ds 
+ \sum_{i=1}^n\frac{1}{L^i}\right) \leq&\,
-\int_{\mathcal{N}_t} k_{s}^2\,ds 
+ C\left(\int_{\mathcal{N}_t}k^2\,ds\right)^3 
+ C\sum_{i=1}^n\frac{\left(\int_{\mathcal{N}_t}k^2\,ds\right)^2} {L^i}\nonumber\\ 
&\,+C\sum_{i=1}^n\frac{\left(\int_{\mathcal{N}_t}k^2\,ds\right)^{3/2}} {(L^i)^{3/2}}+ C\sum_{i=1}^n\frac{1}{(L^i)^3} + C\\ 
\leq&\, C\left(\int_{\mathcal{N}_t}k^2\,ds\right)^3 
+ C\sum_{i=1}^n\frac{1}{(L^i)^3}+C\nonumber\\ 
\leq&\, C\left(\int_{\mathcal{N}_t}k^2\,ds 
+\sum_{i=1}^n\frac{1}{L^i}+1\right)^3\,,
\end{align} 
with a constant $C$ depending only on the structure of the network 
and on the constants $C_0$ and $C_1$ in 
assumption~\eqref{endsmooth}.

\subsection{Consequences of the estimates}

Thanks to the just computed estimates on the curvature and on the inverse of the
length one can obtain the following result:

\begin{prop}\label{stimaL} 
For every $M>0$ there exists a time $T_M\in (0,T)$, depending only on the {\em structure} of the network and on the constants $C_0$ and $C_1$ in assumption~\eqref{endsmooth}, such that if the square of the $L^2$ norm of the curvature and the inverses of the lengths of the curves of $\mathcal{N}_0$ are bounded by $M$, then the square of the $L^2$ norm of $k$ and the inverses of the lengths of the curves of $\mathcal{N}_t$ are smaller than $2(n+1)M+1$, for every time $t\in[0,T_M]$. 
\end{prop}

\begin{proof} 
Consider the  positive function 
$f(t)=\int_{\mathcal{N}_t} k^2\,ds 
+\sum_{i=1}^n\frac{1}{L^i(t)}+1$. 
Then by  inequality~\eqref{contistimaL}
$f$ satisfies the differential inequality $f^\prime\leq Cf^3$.
After integration it reads as
$$
f^2(t)\leq \frac{f^2(0)}{1-2Ctf^2(0)}\leq \frac{f^2(0)}{1-2Ct[(n+1)M+1]}\,,
$$
then if $t\leq T_M=\frac{3}{8C[(n+1)M+1]}$ we get
$f(t)\leq2f(0)$. 
Hence
$$
\int_{\mathcal{N}_t} k^2\,ds +\sum_{i=1}^n\frac{1}{L^i(t)}\leq 2\int_{\mathcal{N}_0} k^2\,ds +2\sum_{i=1}^n\frac{1}{L^i(0)}+1\leq 2[(n+1)M]+1\,.
$$ 
\end{proof}

The combination of these estimates
implies estimates on all the derivatives
of the maps $\gamma^i$, stated in the next proposition.

\begin{prop}\label{unif222} 
If $\mathcal{N}_t$ is a $C^\infty$ special evolution of the initial network 
$\mathcal{N}_0=\bigcup_{i=1}^n\sigma^i$, satisfying assumption~\eqref{endsmooth}, such that the lengths  of the $n$ curves are uniformly bounded away from zero 
and the   $L^2$ norm of the curvature is uniformly bounded by some
constants in the time interval $[0,T)$, then 
\begin{itemize} 
\item all the derivatives in space and time of $k$ and $\lambda$ are
  uniformly bounded in $[0,1]\times[0,T)$, 
\item all the derivatives in space and time of the curves  $\gamma^i(t,x)$ 
are uniformly bounded in $[0,1]\times[0,T)$, 
\item the quantities $\vert\gamma^i_x(t,x)\vert$ are uniformly bounded   
from above and away from zero in $[0,1]\times[0,T)$. 
\end{itemize} 
All the bounds depend only on the uniform controls on the $L^2$ norm of $k$, on the
lengths of the curves of the network from below, on the constants
$C_j$ in assumption~\eqref{endsmooth}, on the $L^\infty$ norms of
the derivatives of the curves $\sigma^i$ 
and on the bound from above and below on $\vert\sigma^i_x(t,x)\vert$, for the curves describing 
the initial network $\mathcal{N}_0$. 
\end{prop}

By means of Proposition~\ref{stimaL} we can strengthen the conclusion of
Proposition~\ref{unif222}.

\begin{cor}\label{topolino7} In the hypothesis of the previous proposition, in the time
  interval $[0,T_M]$ all the bounds in Proposition~\ref{unif222} depend only on the 
  $L^2$ norm of $k$ on $\mathcal{N}_0$, on the constants $C_j$ in assumption~\eqref{endsmooth}, 
  on the $L^\infty$ norms of the derivatives of the curves $\sigma^i$, 
  on the bound from above and below on $\vert\sigma^i_x(t,x)\vert$ and on the lengths of the curves of 
  the initial network $\mathcal{N}_0$.
\end{cor}

By means of the a priori estimates we can work out some results
about the smooth flow of an initial regular geometrically smooth
network $\mathcal{N}_0$.

\begin{thm}\label{curvexplod} 
If $[0,T)$, with $T<+\infty$, is the maximal time interval of existence of a $C^\infty$ curvature flow 
of an initial geometrically smooth network $\mathcal{N}_0$,  then 
\begin{enumerate} 
\item either the inferior limit of the length of at least one curve of $\mathcal{N}_t$ is zero, as $t\to T$, 
\item or $\limup_{t\to T}\int_{\mathcal{N}_t}k^2\,ds=+\infty$. 
\end{enumerate} 
\end{thm}
\begin{proof}
We can $C^\infty$ reparametrize the flow $\mathcal{N}_t$ in order that it becomes a special smooth flow $\widetilde{\mathcal{N}}_t$ in $[0,T)$.
If the lengths of the curves of $\mathcal{N}_t$ are uniformly bounded away from zero and the $L^2$ norm of $k$ is bounded, the same holds for the networks $\widetilde{\mathcal{N}}_t$.
 Then, by Proposition~\ref{unif222} and
Ascoli--Arzel\`a Theorem, the network $\widetilde{\mathcal{N}}_t$ converges in $C^\infty$ to a smooth network $\widetilde{\mathcal{N}}_T$ as $t\to
T$. 
We could hence restart the flow obtaining a $C^\infty$ special curvature flow in a longer time interval. Reparametrizing back this last flow, we get a $C^\infty$   ``extension'' in time of the flow $\mathcal{N}_t$, hence  contradicting the maximality of the interval
$[0,T)$.
\end{proof}

\begin{prop}\label{curvexplod2}
If $[0,T)$, with $T<+\infty$, is the maximal time interval of existence of a $C^\infty$ curvature flow 
of an initial geometrically smooth network $\mathcal{N}_0$.
If the lengths of the $n$ curves are uniformly positively bounded from below, 
then this superior limit is actually a limit and  
there exists a positive constant $C$ such that
$$
\int_{{\mathcal{N}_t}} k^2\,ds \geq \frac{C}{\sqrt{T-t}}\,,
$$
for every $t\in[0, T)$.
\end{prop}

\begin{proof}
Considering the flow $\widetilde{\mathcal{N}}_t$ introduced
in the previous theorem. 
By means of differential inequality~\eqref{evolint999}, we have
$$
\frac{d\,}{dt} \int_{{\widetilde{\mathcal{N}}_t}} \widetilde{k}^2\,ds
\leq C \left(\int_{{\widetilde{\mathcal{N}}_t}} \widetilde{k}^2\,ds\right)^{3} + C
\leq C \left(1+\int_{{\widetilde{\mathcal{N}}_t}} \widetilde{k}^2\,ds\right)^{3}\,,
$$
which, after integration between $t,r\in[0,T)$ with $t<r$, gives
$$
\frac{1}{\left(1+\int_{{\widetilde{\mathcal{N}}_t}} \widetilde{k}^2\,ds\right)^{2}}
-\frac{1}{\left(1+\int_{{\widetilde{\mathcal{N}}_r}} \widetilde{k}^2\,ds\right)^{2}}\leq C(r-t)\,.
$$
Then, if case $(1)$ does not hold, we can choose 
a sequence of times $r_j\to T$ such that 
$\int_{\widetilde{\mathcal{N}}_{r_j}} \widetilde{k}^2\,ds\to+\infty$. Putting $r=r_j$ in the inequality above and passing
to the limit, as $j\to\infty$, we get
$$
\frac{1}{\left(1+\int_{{\widetilde{\mathcal{N}}_t}} \widetilde{k}^2\,ds\right)^{2}}\leq C(T-t)\,,
$$
hence, for every $t\in[0, T)$,
$$
\int_{{\widetilde{\mathcal{N}}_t}} \widetilde{k}^2\,ds \geq \frac{{C}}{\sqrt{T-t}}-1\geq
\frac{C}{\sqrt{T-t}}\,,
$$
for some positive constant $C$ and $\lim_{t\to T}\int_{{\widetilde{\mathcal{N}}_t}}k^2\,ds=+\infty$.\\
By the invariance of the curvature by reparametrization, this last estimate implies the same estimate for the flow $\mathcal{N}_t$.
\end{proof}

This theorem obviously implies the following corollary.

\begin{cor}\label{kexplod} If $[0,T)$, with $T<+\infty$, is the maximal time interval of existence of a $C^\infty$ curvature flow of an initial geometrically smooth network $\mathcal{N}_0$ and the lengths of the curves are uniformly bounded away from zero, then
\begin{equation}\label{krate}
\max_{\mathcal{N}_t}k^2\geq\frac{C}{\sqrt{T-t}}\to+\infty\,,
\end{equation}
as $t\to T$.
\end{cor}

In the case of the evolution $\gamma_t$ 
of a single closed curve in the plane 
there exists a constant $C>0$ such that if at time $T>0$ a 
singularity develops, then 
$$
\max_{{\gamma}_t} k^2\geq\frac{C}{{T-t}}
$$
for every $t\in[0,T)$ (see~\cite{huisk3}).
It is unknown 
if this lower bound on the rate of blow-up of the curvature holds
also in the case of the evolution of a network.

\begin{rem}
Using more refine estimates it is possible to weaken the assumption of 
Theorem~\ref{curvexplod}:
one can suppose to have a $C^{\frac{2+\alpha}{2},2+\alpha}$ curvature flow
(see~\cite[p.~33]{mannovplusch}).
\end{rem}

We conclude this section with the following estimate from below on the maximal time of smooth existence.

\begin{prop}\label{unif333} For every $M>0$ there exists a
  positive time $T_M$ such that if the $L^2$ norm of the curvature and
  the inverses of the lengths of the geometrically smooth network $\mathcal{N}_0$ are bounded by $M$, then the  maximal time of existence $T>0$ of a $C^\infty$ curvature flow of $\mathcal{N}_0$ is larger than $T_M$.
\end{prop}
\begin{proof}
As before, considering again the reparametrized special curvature flow $\widetilde{\mathcal{N}}_t$, by Proposition~\ref{stimaL} in the interval 
$[0,\min\{T_M,T\})$ the
$L^2$ norm of $\widetilde{k}$ and the inverses of the lengths of the curves
of $\widetilde{\mathcal{N}}_t$ are bounded by $2M^2+6M$.\\
Then, by Theorem~\ref{curvexplod}, the value $\min\{T_M,T\}$ cannot
coincide with the maximal time of existence of $\widetilde{\mathcal{N}}_t$ (hence of $\mathcal{N}_t$), so it must be $T>T_M$.
\end{proof} 

\section{Analysis of singularities}\label{singular}

\subsection{Huisken's Monotonicity Formula}\label{monotonsec}

We shall use the following notation for the evolution of a
network in $\Omega\subset\R^2$:
let $\mathcal{N}\subset\R^2$ be a network homeomorphic
to the all $\mathcal{N}_t$, we consider a map
$$
F:(0,T)\times\mathcal{N}\to\mathbb{R}^2
$$ given by the union of the maps
$\gamma^i:(0,T)\times I_i\to\overline{\Omega}$ 
(with $I_i$ the intervals $[0,1], (0,1], [1,0)$ or $(0,1)$)
describing the curvature flow of the network in the time
interval $(0,T)$, that is $\mathcal{N}_t=F(t,\mathcal{N})$.

\medskip

Let us start from the easiest case in which the network is composed by a unique closed
simple smooth curve.
Let $t_0 \in (0,+\infty), x_0\in\R^2$ and
$\rho_{t_0,x_0}:[0,t_0)\times\mathbb{R}^2$ be the one--dimensional 
\textbf{backward  heat kernel} in $\R^2$ relative to $(t_0,x_0)$, that is
$$
\rho_{t_0,x_0}(t,x)=\frac{e^{-\frac{\vert x-x_0\vert^2}{4(t_0-t)}}}{\sqrt{4\pi(t_0-t)}}\,.
$$

\begin{thm}[Monotonicity Formula]\label{promono1}
Assume $t_0>0$. For every $t\in [0,\min\{t_0,T\})$ and $x_0\in\R^2$ we have 
\begin{align}
\frac{d\,}{dt}\int_{{\mathcal{N}_t}} \rho_{t_0,x_0}(t,x)\,ds= &\,
-\int_{{\mathcal{N}_t}} \left\vert\,\boldsymbol{k}
+\frac{(x-x_0)^{\perp}}{2(t_0-t)}\right\vert^2 \rho_{t_0,x_0}(t,x)\,ds\label{eqmonfor1}
\end{align} 
\end{thm}
\begin{proof}
See~\cite{huisk3}. 
\end{proof}

Then one can wonder if a modified version of this formula holds
for networks. Clearly one needs a way to deal with the boundary points
(the triple junctions).
In~\cite{mannovtor} the authors gave a positive answer to this question 
in the case of a triod.
With a  slight modification of the computation in~\cite[Lemma~6.3]{mannovtor}
one can extend the result to any regular network.
As before,
with a little abuse of notation, we will write $\tau(t,P^r)$ and
$\lambda(t,P^r)$ respectively for the unit tangent vector and the
tangential velocity at the end--point $P^r$ of the curve of the
network getting at such point, for any $r\in\{1,2,\dots,l\}$.

\begin{prop}[Monotonicity Formula]\label{promono}
Assume $t_0>0$. For every $t\in [0,\min\{t_0,T\})$ and $x_0\in\R^2$ the following identity holds 
\begin{align}
\frac{d\,}{dt}\int_{{\mathcal{N}_t}} \rho_{t_0,x_0}(t,x)\,ds= &\,
-\int_{{\mathcal{N}_t}} \left\vert\,\boldsymbol{k}
+\frac{(x-x_0)^{\perp}}{2(t_0-t)}\right\vert^2 \rho_{t_0,x_0}(t,x)\,ds\label{eqmonfor}\\ &\,
+\sum_{r=1}^l\biggl[\biggl\langle\,\frac{P^r-x_0}{2(t_0-t)}\,\biggr\vert\,
\tau(t,P^r)\,\biggr\rangle -\lambda(t,P^r)\,\biggl]\,
\rho_{t_0,x_0}(t,P^r)\,.\nonumber 
\end{align} 
\end{prop}
Integrating between $t_1$ and $t_2$ with $0\leq t_1\leq t_2<\min\{t_0,T\}$ we get 
\begin{align*} 
\int_{t_1}^{t_2}\int_{{\mathcal{N}_t}} \left\vert\,
\boldsymbol{k}+\frac{(x-x_0)^{\perp}}{2(t_0-t)}\right\vert^2 &\rho_{t_0,x_0}(t,x)\,ds\,dt 
=  \,\int_{{\mathcal{N}_{t_1}}} \rho_{t_0,x_0}(x,t_1)\,ds - 
\int_{{\mathcal{N}_{t_2}}} \rho_{t_0,x_0}(x,t_2)\,ds\\ 
&\,+\sum_{r=1}^l\int_{t_1}^{t_2} \biggl[\biggl\langle\,\frac{P^r-x_0}{2(t_0-t)}\,\biggr\vert\,
\tau(t,P^r)\,\biggr\rangle -\lambda(t,P^r)\,\biggl]\,
\rho_{t_0,x_0}(t,P^r)\,dt\,.
\end{align*} 
We need the following lemma in order to estimate the end--points contribution
(see~\cite[Lemma~6.5]{mannovtor}).

\begin{lem}\label{stimadib}  
For every $r\in\{1,2,\dots,l\}$ and $x_0\in\R^2$, the following estimate holds  
\begin{equation*} 
\left\vert\int_{t}^{t_0}\biggl[\biggl\langle\,\frac{P^r-x_0}{2(t_0-\xi)}\,\biggr\vert\,
\tau(\xi,P^r)\,\biggr\rangle -\lambda(\xi,P^r)\,\biggl]\,
\rho_{t_0,x_0}(\xi,P^r)\,d\xi\,\right\vert\leq C\,, 
\end{equation*} 
where $C$ is a constant depending only 
on the constants $C_l$ in assumption~\eqref{endsmooth}.\\
Then for every point $x_0\in\R^2$, we have 
\begin{equation*} 
\lim_{t\to t_0}\sum_{r=1}^l\int_{t}^{t_0}\biggl[\biggl\langle\,\frac{P^r-x_0}{2(t_0-\xi)}\,\biggr\vert\,
\tau(\xi,P^r)\,\biggr\rangle -\lambda(\xi,P^r)\,\biggl]\,
\rho_{t_0,x_0}(\xi,P^r)\,d\xi=0\,. 
\end{equation*} 
\end{lem}

As a consequence, the following definition is well posed.

\begin{defn}[Gaussian densities]\label{Gaussiandensities}
For every $t_0 \in(0,+\infty),\,x_0\in\R^2$ we define the 
\textbf{Gaussian density} function $\Theta_{t_0,x_0}:[0,\min\{t_0,T\})\to\R$ as 
$$
\Theta_{t_0,x_0}(t)=\int_{\mathcal{N}_t}\rho_{t_0,x_0}(t,\cdot)\,ds
$$
and provided $t_0\leq T$ the \textbf{limit Gaussian density} function 
$\widehat{\Theta}:(0,+\infty)\times\mathbb{R}^2\to\R$ as
\[
\widehat{\Theta}(t_0,x_0)=\lim_{t\to t_0}\Theta_{t_0,x_0}(t)\,.
\]
\end{defn}
For every $(t_0,x_0)\in(0,T]\times\mathbb{R}^2$, the limit $\widehat{\Theta}(t_0,x_0)$ exists
(by the monotonicity of $\Theta_{t_0,x_0}$) it is finite and non negative. 
Moreover the map $\widehat{\Theta}:\mathbb{R}^2\to\mathbb{R}$ is upper 
semicontinuous~\cite[Proposition~2.12]{MMN13}.

\subsection{Dynamical rescaling} 

We introduce the rescaling procedure of Huisken in~\cite{huisk3} at the maximal time $T$.\\ 
Fixed $x_0\in\R^2$, 
let $\widetilde{F}_{x_0}:[-1/2\log{T},+\infty)\times \mathcal{N} \to\R^2$ 
be the map 
$$ 
\widetilde{F}_{x_0}(\tt,p)
=\frac{F(t,p)-x_0}{\sqrt{2(T-t)}}\qquad \tt(t)
=-\frac{1}{2}\log{(T-t)} 
$$ 
then, the rescaled networks are given by  
\begin{equation}\label{huiskeqdef}
\widetilde{\mathcal{N}}_{\tt,x_0}=\frac{\mathcal{N}_t-x_0}{\sqrt{2(T-t)}}
\end{equation}
and they evolve according to the equation  
$$ 
\frac{\partial\,}{\partial   \tt}\widetilde{F}_{x_0}(\tt,p)
=\widetilde{\boldsymbol{v}}(\tt,p)+\widetilde{F}_{x_0}(\tt,p) 
$$
where  
$$ \widetilde{\boldsymbol{v}}(\tt,p)=\sqrt{2(T-t(\tt))}\cdot\boldsymbol{v}(t(\tt),p)=
\widetilde{k}\nu+\widetilde{\lambda}\tau\qquad 
\text{ and }\qquad t(\tt)=T-e^{-2\tt}\,. 
$$ 
Notice that we did not put the sign `` $\,\widetilde{}\,$ " over the unit tangent and normal, 
since they remain the same after the rescaling.\\ 
When there is no ambiguity on the point $x_0$,
we will write 
$\widetilde{P}^r(\tt)=\widetilde{F}_{x_0}(\tt, P^r)$ 
for the end--points
of the rescaled network $\widetilde{\mathcal{N}}_{\tt,x_0}$.\\ 
The rescaled curvature evolves according to the following equation, 
\begin{equation*}\label{evolriscforf} 
{\partial_\tt} \widetilde{k}= \widetilde{k}_{\ssss\ssss}+\widetilde{k}_\ssss\widetilde{\lambda}
+ \widetilde{k}^3 -\widetilde{k} 
\end{equation*} 
which can be obtained by means of the commutation law 
\begin{equation*}\label{commutforf} 
{\partial_\tt}{\partial_\ssss}={\partial_\ssss}{\partial_\tt} 
+ (\widetilde{k}^2 -\widetilde{\lambda}_\ssss-1){\partial_\ssss}\,, 
\end{equation*} 
where we denoted with $\ssss$ the arclength parameter for $\widetilde{\mathcal{N}}_{\tt,x_0}$.

By straightforward computations (see~\cite{huisk3})
we have the following rescaled version of the Monotonicity Formula.

\begin{prop}[Rescaled Monotonicity Formula]\label{rescaledmonformula}
Let $x_0\in\R^2$ and set   
$$ \widetilde{\rho}(x)=e^{-\frac{\vert x\vert^2}{2}} $$ 
For every $\tt\in[-1/2\log{T},+\infty)$ the following identity holds 
\begin{equation*} 
\frac{d\,}{d\tt}\int_{\widetilde{\mathcal{N}}_{\tt,x_0}} \widetilde{\rho}(x)\,d\ssss= 
-\int_{\widetilde{\mathcal{N}}_{\tt,x_0}}\vert \,
\widetilde{\boldsymbol{k}}+x^\perp\vert^2\widetilde{\rho}(x)\,d\ssss 
+\sum_{r=1}^l\Bigl[\Bigl\langle\,{\widetilde{P}^r(\tt)} \,\Bigl\vert\,
{\tau}(t(\tt),P^r)\Bigr\rangle-\widetilde{\lambda}(\tt,P^r)\Bigl]\,
\widetilde{\rho}(\widetilde{P}^r(\tt))\label{reseqmonfor} 
\end{equation*} 
where $\widetilde{P}^r(\tt)=\frac{P^r-x_0}{\sqrt{2(T-t(\tt))}}$.\\  
Integrating between $\tt_1$ and $\tt_2$ with  $-1/2\log{T}\leq \tt_1\leq \tt_2<+\infty$ we get 
\begin{align} 
\int_{\tt_1}^{\tt_2}\int_{\widetilde{\mathcal{N}}_{\tt,x_0}}\vert \,\widetilde{\boldsymbol{k}}
+x^\perp\vert^2\widetilde{\rho}(x)\,d\ssss\,d\tt= &\, 
\int_{\widetilde{\mathcal{N}}_{\tt_1,x_0}}\widetilde{\rho}(x)\,d\ssss 
-\int_{\widetilde{\mathcal{N}}_{\tt_2,x_0}}\widetilde{\rho}(x)\,d\ssss\label{reseqmonfor-int}\\ &\,
+\sum_{r=1}^l\int_{\tt_1}^{\tt_2}\Bigl[\Bigl\langle\,\widetilde{P}^r(\tt)\,
\Bigl\vert\,{\tau}(t(\tt),P^r)\Bigr\rangle-\widetilde{\lambda}(\tt,P^r)\Bigl]\,
\widetilde{\rho}(\widetilde{P}^r(\tt))\,d\tt\,.\nonumber 
\end{align} 
\end{prop}

We have also the analog of Lemma~\ref{stimadib} 
(see~\cite[Lemma~6.7 ]{mannovtor}).

\begin{lem}\label{rescstimadib} 
For every $r\in\{1,2,\dots,l\}$ and $x_0\in\R^2$,
the following estimate holds for all $\tt\in \bigl[-\frac{1}{2}\log{T},+\infty\bigr)$,
\begin{equation*} 
\left\vert\int_{\tt}^{+\infty}\Bigl[\Bigl\langle\,\widetilde{P}^r(\xi)\,
\Bigl\vert\,{\tau}(t(\xi),P^r)\Bigr\rangle-\widetilde{\lambda}(\xi,P^r)\Bigl]\,d\xi\,\right\vert\leq C\,,
\end{equation*}
where $C$ is a constant depending only on the 
constants $C_l$ in assumption~\eqref{endsmooth}.\\
As a consequence, for every point $x_0\in\R^2$, we have
\begin{equation*} 
\lim_{\tt\to  +\infty}\sum_{r=1}^l\int_{\tt}^{+\infty}\Bigl[\Bigl\langle\,\widetilde{P}^r(\xi)\,
\Bigl\vert\,{\tau}(t(\xi),P^r)\Bigr\rangle-\widetilde{\lambda}(\xi,P^r)\Bigl]\,d\xi=0\,.
\end{equation*}
\end{lem}

\subsection{Blow--up limits}

We now discuss the possible blow--up limits of an evolving
network at the maximal time of existence. 
This analysis can be seen as a tool to exclude 
the possible arising of singularity in the evolution and to obtain
(if possible)  global existence of the flow.

\medskip

Thanks to Theorem~\ref{curvexplod} we know what happens
when the evolution approaches the singular time $T$:
either the length of at least one curve  of the network
goes to zero, or the $L^2$-norm of the
curvature blows-up.  
When the curvature does not remain bounded, 
we look at the possible limit
networks after (Huisken's dynamical) rescaling procedure.
The  rescaled Monotonicity Formula~\ref{rescaledmonformula}
will play a crucial role. 
We first suppose that the length of all the curves of the network
remains strictly positive during the evolution. 
In this case the classification of the limits is complete (Proposition~\ref{resclimit}).
Without a bound from below on the length of the curves  the
situation is more involved, and
we will see that
in general the limit sets are no longer regular networks.
For this purpose, we shall
introduce the notion of degenerate regular network.

\medskip 

We now describe the blow--up limit of networks under the assumption that the
length of each curve is bounded below by a positive constant independent of time.
We start with a lemma
due to A.~Stone~\cite{stone1}.

\begin{lem}\label{resclimit2H}
Let $\widetilde{\mathcal{N}}_{\tt,x_0}$ be the family of rescaled networks  obtained via Huisken's dynamical procedure around some $x_0\in\R^2$ as defined in formula~\eqref{huiskeqdef}.
\begin{enumerate}
\item There exists a constant $C=C(\mathcal{N}_0)$ such that, for every $\overline{x},\in\R^2$, $\tt\in\bigl[-\frac{1}{2}\log T,+\infty\bigr)$ and $R>0$ there holds
\begin{equation}\label{equ10bissH}
{\mathcal H}^1(\widetilde{\mathcal{N}}_{\tt,x_0}\cap B_R(\overline{x}))\leq CR\,.
\end{equation}
\item For any $\varepsilon > 0$ there is a uniform radius $R= R(\varepsilon)$ such that
\begin{equation*}
\int_{\widetilde{\mathcal{N}}_{\tt,x_0}\setminus B_R(\overline{x})} e^{-|x|^2 /2}\,ds\leq \varepsilon\,,
\end{equation*}
that is, the family of measures $e^{-|x|^2 /2}\,{\mathcal H}^1\res\widetilde{\mathcal{N}}_{\tt,x_0}$ is {\em tight} (see~\cite{dellame}).
\end{enumerate}
\end{lem}
 
\begin{prop}\label{resclimit}
Let $\mathcal{N}_t=\bigcup_{i=1}^n\gamma^i(t,[0,1])$
be a $C^{1,2}$ curvature flow of regular networks
 with fixed end--points in a smooth, strictly convex, bounded open set $\Omega\subset\R^2$
  in the time interval $[0,T)$.
Assume that the lengths $L^i(t)$ of the curves of the networks are uniformly in time bounded
away from zero
for every $i\in\{1,2,\dots, n\}$. 
Then
for every $x_0\in\R^2$ and for every subset $\mathcal I$ of  $[-1/2\log
T,+\infty)$ with infinite Lebesgue measure, 
there exists a sequence of rescaled times
$\tt_j\to+\infty$, with $\tt_j\in{\mathcal I}$, such that the sequence
of rescaled networks $\widetilde{\mathcal{N}}_{\tt_{j},x_0}$ (obtained via Huisken's dynamical procedure) 
converges in $C^{1,\alpha}\loc\cap W^{2,2}\loc$, for any $\alpha \in (0,1/2)$,
 to a (possibly empty) limit, which is (if non-empty)
\begin{itemize}
\item a straight line through the origin with multiplicity $m\in\NN$ (in this case $\widehat\Theta(x_0)=m$);
\item a standard triod centered at the origin with multiplicity
  $1$ (in this case $\widehat\Theta(x_0)=3/2$).
\item a halfline from the origin with multiplicity $1$ (in this case $\widehat\Theta(x_0)=1/2$).
\end{itemize}
Moreover the $L^2$--norm of the curvature of $\widetilde{\mathcal{N}}_{\tt_{j},x_0}$ goes to zero in every ball $B_R\subset\R^2$, as $j\to\infty$. 
\end{prop}

\begin{proof}
We divide the proof into three steps.
We take for simplicity $x_0=0$.

\medskip

\textit{Step 1: Convergence to $\widetilde{\mathcal{N}}_\infty$.}\\
Consider the  rescaled Monotonicity Formula~\eqref{reseqmonfor-int} and 
let $\tt_1=-1/2\log T$ and $\tt_2\to +\infty$.
Then thanks to Lemma~\ref{rescstimadib} we get
$$
\int\limits_{-1/2\log{T}}^{+\infty}\int\limits_{\widetilde{\mathcal{N}}_{\tt,x_0}}
\vert\,\widetilde{\boldsymbol{k}}+x^\perp\vert^2\widetilde{\rho}\,d\sigma\,d\tt<+\infty\,,
$$
which implies
$$
\int\limits_{{\mathcal{I}}}\int\limits_{\widetilde{\mathcal{N}}_{\tt,x_0}}
\vert\,\widetilde{\boldsymbol{k}}+x^\perp\vert^2\widetilde{\rho}\,d\sigma\,d{\tt}<+\infty\,.
$$
Being the last integral finite and being the integrand a non negative
function on a set of infinite Lebesgue measure, we can extract within ${\mathcal I}$ a
sequence of times $\tt_{j}\to+\infty$, such that 
\begin{equation}\label{eqW22}
\lim_{j\to +\infty}\int\limits_{\widetilde{\mathcal{N}}_{\tt_j,x_0}}
\vert\,\widetilde{\boldsymbol{k}}+x^\perp\vert^2\widetilde{\rho}\,d\sigma
=0\,.
\end{equation} 
It follows that for every ball $B_R$ of radius $R$ in $\R^2$ the networks $\widetilde{\mathcal{N}}_{\tt_j,x_0}$ have
curvature uniformly bounded in $L^2(B_R)$. 
Moreover, by Lemma~\ref{resclimit2H}, for every ball 
$B_R$ centered at the origin of $\R^2$ we have the uniform bound 
$\HH^1({\widetilde{\mathcal{N}}_{\tt_j,x_0}}\cap B_R)\leq C R$, for some
constant $C$ independent of $j\in\NN$. Then reparametrizing the
rescaled networks by arclength, we obtain curves with uniformly
bounded first derivatives and with second derivatives
uniformly bounded in $L^2\loc$. \\
By a standard compactness argument (see~\cite{huisk3,langer2}), 
the sequence ${\widetilde{\mathcal{N}}_{\tt_{j},x_0}}$ of reparametrized
networks admits a subsequence ${\widetilde{\mathcal{N}}_{\tt_{j_l},x_0}}$
which converges, weakly in $W^{2,2}\loc$ and strongly in $C^{1,\alpha}\loc$, 
to a (possibly empty) limit 
$\widetilde{\mathcal{N}}_\infty$ 
(possibly with multiplicity). 
The strong convergence in $W^{2,2}\loc$ is implied by the weak convergence
in $W^{2,2}\loc$ and equation~\eqref{eqW22}.

\medskip

\textit{Step 2: The limit $\widetilde{\mathcal{N}}_\infty$ is a regular shrinker.}\\
We first notice that
 the bound from below on the lengths
prevents any ``collapsing'' along the rescaled sequence.
Since the integral functional 
$$
\widetilde{\mathcal{N}}\mapsto
\int\limits_{\widetilde{\mathcal{N}}}\vert\,\widetilde{\boldsymbol{k}}+x^\perp\vert^2\widetilde{\rho}\,d\sigma
$$
is lower semicontinuous with respect to this convergence
 (see~\cite{simon}, for instance), the limit ${\widetilde{\mathcal{N}}}_\infty$ satisfies 
$\widetilde{\boldsymbol{k}}_\infty+x^\perp=0$ in the sense of distributions. \\
A priori, the limit network is composed by curves in
$W^{2,2}\loc$, but from the relation
$\widetilde{\boldsymbol{k}}_\infty+x^\perp=0$, it follows that 
the curvature $\widetilde{\boldsymbol{k}}_\infty$ is continuous. 
By a bootstrap argument, it is then easy to see that $\widetilde{\mathcal{N}}_\infty$ 
is actually composed by $C^\infty$ curves.

\medskip

\textit{Step 3: Classification of the possible limits.}\\
If the point $x_0\in\R^2$ is distinct from all the end--points $P^r$, 
then $\widetilde{\mathcal{N}}_\infty$ has no end--points, since they go to infinity 
along the rescaled sequence. If $x_0=P^r$ for some $r$, the set $\widetilde{\mathcal{N}}_\infty$ has a single end--point at the
origin of $\R^2$.\\ 
Moreover, from the lower bound on the length of the original curves it follows
that all the curves of $\widetilde{\mathcal{N}}_\infty$ have infinite
length, hence, by Remark~\ref{abla}, they must be pieces of straight
lines from the origin.\\
This implies that every connected component of the graph underlying 
$\widetilde{\mathcal{N}}_\infty$ can contain at most one $3$--point 
and in such case such component must be  a standard triod
(the $120$ degrees condition must be satisfied) with multiplicity one
since the  converging networks are all embedded 
(to get in the $C^1\loc$--limit a triod with multiplicity higher than one it is
necessary that the approximating networks have
self--intersections). Moreover, since the converging networks
are embedded, if both a standard triod and a straight line or
another triod are present,  they would intersect
transversally. Hence if a standard triod is present, a straight line cannot be 
present and conversely if a straight line
is present, a triod cannot be present.

If no end--point is present, that is, we are rescaling around a
point in $\Omega$ (not on its boundary), and no $3$--point is 
present, the only possibility is a straight line (possibly with
multiplicity) through the origin.

If an end--point is present, we are rescaling around an end--point of
the evolving network, hence, by the convexity of $\Omega$ (which
contains all the networks) the limit $\widetilde{\mathcal{N}}_\infty$ must be
contained in a halfplane with boundary a straight line $H$ for the
origin. This exclude the presence of a standard triod since it cannot
be contained in any halfplane. Another halfline is obviously excluded,
since they ``come'' only from end--points and they are all
distinct. In order to exclude the presence of a straight line, we observe that the argument of Proposition~\ref{omegaok2} implies that, if $\Om_t\subset \Om$ is the evolution by curvature of $\partial \Om$ keeping fixed the end--points $P^r$, the blow--up of $\Om_t$ at an end--point must be a cone spanning angle strictly less then $\pi$ 
(here we use the fact that three end--points are not aligned) and $\widetilde{\mathcal{N}}_\infty$ is contained in such a cone. It follows that $\widetilde{\mathcal{N}}_\infty$ cannot contain a straight line.

In every case the curvature of $\widetilde{\mathcal{N}}_\infty$ is zero everywhere and the
last statement follows by the $W^{2,2}\loc$--convergence.

\end{proof}

\begin{rem}
In the previous proposition the hypothesis on the length of the curve can be replace
by the weaker assumption that
the lengths $L^i(t)$ of the curves satisfy
\begin{equation}
\lim_{t\to T}\frac{L^i(t)}{\sqrt{T-t}}=+\infty\,,
\end{equation}
for every $i\in\{1,2,\dots, n\}$. 
\end{rem}

\begin{lem}
Under the assumptions of Proposition \ref{resclimit}, there holds
\begin{equation}\label{gggg3}
\lim_{j\to\infty}\frac{1}{\sqrt{2\pi}}\int_{\widetilde{\mathcal{N}}_{\tt_j,x_0}}
\widetilde{\rho}\,d\sigma=\frac{1}{\sqrt{2\pi}}
\int_{\widetilde{\mathcal{N}}_\infty}\widetilde{\rho}\,d\overline{\sigma}=
\Theta_{\widetilde{\mathcal{N}}_\infty}=\widehat{\Theta}(T,x_0)\,.
\end{equation}
where $d\overline{\sigma}$ denotes the integration with respect to the
canonical measure on
$\widetilde{\mathcal{N}}_\infty$, counting multiplicities
\end{lem}
\begin{proof}
By means of the second point of Lemma~\ref{resclimit2H}, we can pass to the limit in the Gaussian integral and we get
$$
\lim_{j\to\infty}\frac{1}{\sqrt{2\pi}}
\int_{\widetilde{\mathcal{N}}_{\tt_j,x_0}}\widetilde{\rho}\,d\sigma=\frac{1}{\sqrt{2\pi}}
\int_{\widetilde{\mathcal{N}}_\infty}\widetilde{\rho}\,d\overline{\sigma}=\Theta_{\widetilde{\mathcal{N}}_\infty}\,.
$$
Recalling that
$$
\frac{1}{\sqrt{2\pi}}\int_{\widetilde{\mathcal{N}}_{\tt_j,x_0}}\widetilde{\rho}\,d\sigma
=\int_{\mathcal{N}_{t(\tt_j)}}\rho_{T,x_0}(\t(\tt_j),\cdot)\,ds
=\Theta_{x_0}(t(\tt_j))\to\widehat{\Theta}(T,x_0)
$$
as $j\to\infty$, equality~\eqref{gggg3} follows.
\end{proof}

\begin{rem} If the three end--points $P^{r-1}, P^r, P^{r+1}$ are aligned
the argument of Proposition~\ref{omegaok2} does not work and 
we cannot conclude that the only blow--up at $P^r$ is a halfline with multiplicity $1$. 
It could also be possible that a straight line (possibly with higher multiplicity) 
is present.
\end{rem}

We describe now how
Proposition~\ref{resclimit} allows us to obtain a (conditional) global existence result 
when the lengths of all the curves of the networks are strictly positive.

\medskip

Suppose that $T<+\infty$.
As we have assumed that 
the lengths of all the curves of the network are
uniformly positively bounded from below, 
the curvature blows-up as $t\to T$ (Theorem~\ref{curvexplod} and Proposition~\ref{curvexplod2}). 
Performing a Huisken's rescaling at an
interior point $x_0$ of $\Omega$, we obtain as blow--up limit (if not empty)
a standard triod or 
a straight line with
multiplicity $m\in\mathbb{N}$. 
One can argue as in~\cite{MMN13} to show that 
when such limit is a regular triod, the curvature is locally 
bounded around such point $x_0$.
For the case of a straight line, 
if we suppose that {the multiplicity $m$ is equal to $1$},
by White's local regularity theorem~\cite{white1}
we conclude that 
the curvature is bounded  uniformly  in time, in a neighborhood of the point
$x_0$.
If we instead rescale at an end--point $P^r$ we get a halfline.
This case can be treated as above by means of a reflection
argument. 
Indeed for the flow obtain by the union of the original network and
the reflection of this latter,
the point $P^r$ is no more an end--point. A blow--up at  $P^r$ give a
straight line, implying that the curvature is locally bounded
also around $P^r$ as before by White's theorem.

Supposing that the lengths of the curves of the network
are strictly positive and supposing also that
any blow--up limit has multiplicity one,
it follows that the original network $\mathcal{N}_t$
has bounded curvature as $t\to T$.
Hence
$T$ cannot be a singular time, and we have therefore
global existence of the flow. 

\medskip

In the previous reasoning a key point is the hypothesis 
that the blow--ups have multiplicity one.
Unfortunately, for a general regular network, this is still
conjectural and possibly the major open problem in the subject.

\begin{mulone}[\bf{M1}]
Every possible $C^1\loc$--limit of rescalings of networks of the flow is an embedded network with multiplicity one.
\end{mulone}

However, in some special situations one can actually prove {\bf{M1}}.

\begin{prop}\label{m12trips}
If $\Omega$ is strictly convex and the evolving network $\mathcal{N}_t$ has at most two triple junctions, every $C^1\loc$--limit of rescalings of networks of the flow 
is embedded and has multiplicity one.
\end{prop}
\begin{proof}
See~\cite[Section~4,Corollary~4.7]{mannovplu}.
\end{proof}

\begin{prop}
If during the curvature flow of a tree $\mathcal{N}_t$ the triple junctions stay uniformly far from each other and from the end--points, then
every $C^1\loc$--limit of rescalings of networks of the flow 
is embedded and has multiplicity one.
\end{prop}
\begin{proof}
See~\cite[Proposition~14.14]{mannovplusch}.
\end{proof}

\medskip

We now
remove  the hypothesis on the lengths of the curves of the network.
In this case, nothing prevents a length to go to zero in the limit.

In order to describe the possible limits, we introduce the notion of degenerate regular networks.
First of all we define the \textbf{underlying graph}, which is 
an oriented graph $G$
with $n$ edges $E^i$, that can be bounded and unbounded.
Every vertex of $G$ can either have order one (and in this case it is called end--points of $G$)
or order three.

For every edge $E^i$ we introduce an 
orientation preserving homeomorphisms $\varphi^i:E^i\to I^i$
where $I^i$ is the interval $(0,1)$, $[0,1)$, $(0,1]$ or $[0,1]$.
If $E^i$ is a segment, then $I^i=[0,1]$. 
If it is an halfline, we choose $I^i=[0,1)$ or $I^i=(0,1]$.
Notice that the interval $(0,1)$ can only appear if it is associated to 
an unbounded edge $E^i$ without vertices, 
which is clearly a single connected component of $G$.

We then consider a family of $C^1$ parametrizations $\sigma^i:I^i\to\R^2$.
In the case $I^i$ is $(0,1)$, $[0,1)$ or $(0,1]$, 
the map $\sigma^i$ is a regular $C^1$ curve with unit tangent vector $\tau^i$.
If instead $I^i=[0,1]$ the map $\sigma^i$ can be either a regular $C^1$ curve with unit tangent vector $\tau^i$, 
or a constant map (\textbf{degenerate curves}).
In this last case we 
assign a constant unit vector $\tau^i:I^i\to\R^2$ to the curve $\sigma^i$.
At the points $0$ and $1$ of $I^i$ the 
\textbf{assigned exterior unit tangents} are 
$-\tau^i$ and $\tau^i$, respectively.
The exterior unit tangent vectors (real or assigned) at the relative borders of the intervals $I^i$, $I^j$, $I^k$ of the concurring curves $\sigma^i$, $\sigma^j$ $\sigma^k$ have zero sum 
(\textbf{degenerate $120$ degrees condition}).
We require that
the map $\Gamma:G\to\R^2$ given by the union $\Gamma=\bigcup_{i=1}^n(\sigma^i\comp\varphi^i)$ 
is well defined and continuous.

\medskip

We define a \textbf{degenerate regular network $\mathcal{N}$}
as the the union of the sets $\sigma^i(I^i)$.
If one or several edges $E^i$ of $G$ are mapped under the map $\Gamma:G\to\R^2$ to a single point $p\in\R^2$, we call this
sub--network given by the union $G^\prime$ of such edges $E^i$ the \textbf{core} of 
$\mathcal{N}$ at $p$. 

We call multi--points of the degenerate regular network $\mathcal{N}$ 
the images of the vertices 
of multiplicity three of the graph $G$ by the map $\Gamma$ and end--points
of
$\mathcal{N}$ the images of the vertices of multiplicity one of the graph $G$, 
by the map $\Gamma$.
 
A degenerate regular network $\mathcal{N}$
with underlying graph $G$, seen as a subset in $\R^2$,
is a $C^1$ network, not necessarily regular, that can have
end--points and/or unbounded curves. Moreover, self--intersections
and curves with integer multiplicities can be present. Anyway,
at every image of a multi--point of $G$ the sum (possibly with
multiplicities) of the exterior unit tangents is zero.

\begin{defn} We say that a sequence of regular networks $\mathcal{N}_k=\bigcup_{i=1}^{n}\sigma^{i}_k(I^i_k)$ converges in $C^1\loc$ to a degenerate regular network $\mathcal{N}=\bigcup_{j=1}^{l}\sigma^{j}_\infty(I^j_\infty)$ with underlying graph $G=\bigcup_{j=1}^{l}E^j$ if:
\begin{itemize}

\item letting $O^1, O^2,\dots, O^m$ the multi--points of $\mathcal{N}$, for every open set $\Omega\subset\R^2$ with compact closure in $\R^2\setminus\{O^1, O^2,\dots, O^m\}$, the networks $\mathcal{N}_k$ restricted to $\Omega$, for $k$ large enough, are described by families of regular curves which, after possibly reparametrizing them, converge to the family of regular curves given by the restriction of $\mathcal{N}$ to $\Omega$;

\item for every multi--point $O^p$ of $\mathcal{N}$, image of one or more vertices of the graph $G$ (if a core is present), there is a sufficiently small $R>0$ and a graph $\widetilde{G}=\bigcup_{r=1}^{s}F^r$, with edges $F^r$ associated to intervals $J^r$, such that:
\begin{itemize}
\item the restriction of $\mathcal{N}$ to $B_R(O^p)$ is a regular degenerate network described by a family of curves $\widetilde{\sigma}^{r}_\infty:J^r\to\R^2$ with (possibly ``assigned'', if the curve is degenerate) unit tangent $\widetilde{\tau}^r_\infty$,
\item for $k$ sufficiently large, the restriction of $\mathcal{N}_k$ to $B_R(O^p)$ is a regular network with underlying graph $\widetilde{G}$, described by the family of regular curves $\widetilde{\sigma}^{r}_k:J^r\to\R^2$,
\item for every $j$, possibly after reparametrization of the curves, the sequence of maps $J^r\ni x\mapsto\bigl(\widetilde{\sigma}^r_k(x),\widetilde{\tau}^r_k(x)\bigr)$ 
converge in $C^0\loc$ to the maps $J^r\ni x\mapsto\bigl(\widetilde{\sigma}^r_\infty(x),\widetilde{\tau}^r_\infty(x)\bigr)$, for every $r\in\{1,2,\dots,s\}$.
\end{itemize}
\end{itemize}
We will say that $\mathcal{N}_k$ converges to $\mathcal{N}$ in $C^1\loc\cap E$, where $E$  is some function space, if the above curves also converge in the topology of $E$.
\end{defn}

\medskip

Removing the hypothesis on the lengths of the curves, we get that the limit networks are
degenerate regular networks which are homothetically shrinking under the flow.

\begin{prop}\label{resclimit-general}
Let $\mathcal{N}_t=\bigcup_{i=1}^n\gamma^i(t,[0,1])$ be a $C^{1,2}$ curvature flow of regular networks in the time interval $[0,T]$, then,
for every $x_0\in\R^2$ and for every subset $\mathcal I$ of  $[-1/2\log
T,+\infty)$ with infinite Lebesgue measure, 
there exists a sequence of rescaled times
$\tt_j\to+\infty$, with $\tt_j\in{\mathcal I}$, such that the sequence
of rescaled networks $\widetilde{\mathcal{N}}_{\tt_{j},x_0}$ (obtained via Huisken's dynamical procedure) 
converges in $C^{1,\alpha}\loc\cap W^{2,2}\loc$, for any $\alpha \in (0,1/2)$,
 to a (possibly empty) limit network, which is a degenerate regular shrinker $\widetilde{\mathcal{N}}_\infty$ (possibly with multiplicity greater than one).
 
Moreover, we have
\begin{equation}\label{gggg2}
\lim_{j\to\infty}\frac{1}{\sqrt{2\pi}}\int_{\widetilde{\mathcal{N}}_{\tt_j,x_0}}
\widetilde{\rho}\,d\sigma=\frac{1}{\sqrt{2\pi}}
\int_{\widetilde{\mathcal{N}}_\infty}\widetilde{\rho}\,d\overline{\sigma}=
\Theta_{\widetilde{\mathcal{N}}_\infty}=\widehat{\Theta}(T,x_0)\,.
\end{equation}
where $d\overline{\sigma}$ denotes the integration with respect to the
canonical measure on
$\widetilde{\mathcal{N}}_\infty$, counting multiplicities.
\end{prop}

\begin{rem}\label{nonunico}
Notice that the blow--up limit degenerate shrinker obtained by this proposition a priori depends on the chosen sequence of rescaled times $\tt_j\to+\infty$.
\end{rem}

\begin{rem}
Thanks to Proposition~\ref{m12trips},
if the network $\mathcal{N}$ has at most two triple junctions,
the degenerate regular shrinker $\widetilde{\mathcal{N}}_\infty$ 
has multiplicity one.
\end{rem}

Assuming that the length of at least one curve of   $\mathcal{N}_t$ goes to zero, as $t\to T$, 
there are two possible situations:

\begin{itemize}
\item The curvature stays bounded.
\item The curvature is unbounded as $t\to T$.
\end{itemize}

Suppose that the curvature remains bounded
in the maximal time interval $[0,T)$.
As $t\to T$ 
the networks $\mathcal{N}_t$
converge in $C^1$  (up to
reparametrization)  to
a unique limit degenerate regular
network $\widehat{\mathcal{N}}_T$  in $\Omega$.
This network
can be
\textbf{non--regular} seen as a subset of $\R^2$:
multi--points can appear, but anyway the sum of the exterior unit tangent vectors of the concurring curves at
every multi--point must be zero.
Every triple junction  satisfies the angle condition.
The non--degenerate curves of $\widehat{\mathcal{N}}_T$
belong to $C^1\cap W^{2,\infty}$ and they are smooth outside the multi--points
(for the proof see~\cite[Proposition~10.11]{mannovplusch}).

\medskip

We have seen in Section~\ref{evolunghezza}
that if a region is bounded by less than six
curves then its area decreases linearly in time going to zero at $\overline{T}$.
Not only the area goes to zero in a finite time, but 
also the lengths of all the curves that bound the region.
Moreover when
the lengths of all the curves of the loop go to zero, 
then the curvature blows up.
Let us call the loop $\ell$. Combing~\eqref{areauno} with~\eqref{areaevolreg}
there is a positive constant $c$ such that 
$\int_\ell \vert k\vert \,ds\geq c$.
By H\"{o}lder inequality
$$
c\leq \int_\ell \vert k\vert \,ds\leq \left(\int_\ell k^2\,ds\right)^{1/2} L(\ell)^{1/2}\,,
$$
where $L(\ell)$ is the total length of the loop.
Hence 
$$
\Vert k\Vert_{L^2}\geq \frac{c^2}{L(\ell)}\to\infty \quad\text{as}\quad L(\ell)\to 0\,.
$$
Then at time $\overline{T}$ we have a  singularity
where both  the length goes to zero and the curvature explodes.

\medskip

Developing careful a priori estimates of the curvature one can show that 
if two triple junctions collapse into a $4$--point, then the curvature remains bounded 
(see \cite{mannovplusch}).
The interest of this result
relies on the fact that it describes the formation of a ``type zero" singularity:
a singularity due to the change of topology, not to the blow up of the curvature.
This is a new phenomenon with respect to the classical curve shortening flow
and the mean curvature flow more in general.
Thanks to this result it is possible to show that 
given an initial network without loops (a {tree}), if 
Multiplicity--One Conjecture \textbf{M1} is valid, 
then the curvature is uniformly bounded during the flow.
The only possible ``singularities" are given by the collapse of a curve 
with  two triple junctions going to collide.
Moreover in the case of a tree
we are able to show the uniqueness of the blow up limit (see Remark~\ref{nonunico}).

\medskip

Although one can find example of global existence of the flow
(consider for instance an initial triod contained 
in the triangle with vertices its three end--points and with 
all angles  less than 120 degrees)
our analysis underlines the generic presence of singularities.
Then a natural question is if it is possible to 
 go beyond the singularity.
 
 \medskip
 
 There are results on
 the short time existence of the flow for non-regular networks, that is,
networks with multi--points (not only $3$-points),
or networks that do not satisfy the $120$ degrees condition at the $3$-points.
Till now the most general result of this kind is the one by Ilmanen,  
Neves and Schulze~\cite{Ilnevsch}, which provides short time existence of the flow
starting from a non-regular network with \textbf{bounded curvature}.
Notice that the network arising after the collapse of (exactly) two triple junctions 
has bounded curvature, and therefore 
fits with the hypotheses this result.

 \medskip
 
An ambitious project should be constructing a bridge between the analysis of the 
long time behavior of networks moving by curvature
and short time existence results for non-regular initial data:
one can interpret the short time existence results for non-regular data as a ``restarting" theorem
for the flow after the  onset of the first singularity.

\bibliographystyle{amsplain}
\bibliography{Lecture-notes}

\end{document}